\documentclass[reqno,twoside,11pt]{amsart}

\usepackage{cite}
\usepackage{amsmath,mathtools}
\usepackage{amsfonts}
\usepackage{amssymb}
\usepackage{verbatim}
\usepackage{epsfig}
\usepackage{color}
\usepackage{tikz}
\usepackage{caption,  xifthen,subcaption, enumerate,hyperref, accents}
\usetikzlibrary{patterns}

\definecolor{Blue}{rgb}{0,0,1}
\def\blue{\color{Blue}}

\definecolor{Red}{rgb}{1,0,0}
\def\red{\color{Red}}

\definecolor{Green}{rgb}{0,1,0}

\definecolor{Yellow}{rgb}{1,1,0}

\DeclareFontFamily{OT1}{eusb}{} \DeclareFontShape{OT1}{eusb}{m}{n} {<5> <6> <7> <8> <9> <10> <11> <12> <14.4> eusb10}{}
\DeclareMathAlphabet{\eusb}{OT1}{eusb}{m}{n}

\DeclareFontFamily{OT1}{eusm}{} \DeclareFontShape{OT1}{eusm}{m}{n} {<5> <6> <7> <8> <9> <10> <11> <12> <14.4> eusm10}{}
\DeclareMathAlphabet{\eusm}{OT1}{eusm}{m}{n}

\DeclareFontFamily{OT1}{eufm}{} \DeclareFontShape{OT1}{eufm}{m}{n} {<5> <6> <7> <8> <9> <10> <11> <12> <14.4> eufm10}{}
\DeclareMathAlphabet{\mathfrak}{OT1}{eufm}{m}{n}

\DeclareFontFamily{OT1}{fraktura}{}
\DeclareFontShape{OT1}{fraktura}{m}{n} {<5> <6> <7> <8> <9> <10> <11> <12> <13> <14.4> [1.1] eufm10}{}
\DeclareMathAlphabet{\fraktura}{OT1}{fraktura}{m}{n}

\DeclareFontFamily{OT1}{cmfi}{} \DeclareFontShape{OT1}{cmfi}{m}{n} {<5> <6> <7> <8> <9> <10> <11> <12> <13> <14.4> [0.9] cmfi10}{}
\DeclareMathAlphabet{\cmfi}{OT1}{cmfi}{b}{n}

\DeclareFontFamily{OT1}{cmss}{} \DeclareFontShape{OT1}{cmss}{m}{n} {<5> <6> <7> <8> <9> <10> <11> <12> <13> <14.4> cmss10}{}
\DeclareMathAlphabet{\cmss}{OT1}{cmss}{m}{n}

\newtheoremstyle{thm}{1.5ex}{1.5ex}{\itshape\rmfamily}{} {\bfseries\rmfamily}{}{2ex}{}

\newtheoremstyle{def}{1.5ex}{1.5ex}{\slshape\rmfamily}{} {\bfseries\rmfamily}{}{2ex}{}

\newtheoremstyle{rem}{1.3ex}{1.3ex}{\rmfamily}{} {\itshape}
{} {1.5ex}{}

\theoremstyle{thm}
\newtheorem{theorem}{Theorem}[section]
\newtheorem{lemma}[theorem]{Lemma}
\newtheorem{claim}[theorem]{Claim}
\newtheorem{proposition}[theorem]{Proposition}

\newtheorem*{Main Theorem}{Main Theorem.}
\newtheorem{corollary}[theorem]{Corollary}
\newtheorem*{special theorem}{Lindeberg-Feller Theorem for Martingales}

\theoremstyle{def}
\newtheorem{definition}[theorem]{Definition}
\newtheorem{assumption}{Assumption}
\newtheorem{example}{Example}

\theoremstyle{rem}

\numberwithin{equation}{section}

\renewcommand{\section}{\secdef\sct\sect}
\newcommand{\sct}[2][default]{%
\refstepcounter{section}
\addcontentsline{toc}{section}{{\tocsection {}{\thesection}{\!\!\!\!#1\dotfill}}{}}
\vspace{0.7cm}
\centerline{\scshape\thesection.\ #1} \nopagebreak \vspace{0.2cm}}
\newcommand{\sect}[1]{%
\vspace{0.4cm} \centerline{\large\scshape\rmfamily #1}
\vspace{0.2cm}}

\newcommand{\ignore}[1]{{}}

\renewcommand{\subsection}{\secdef\subsct\sbsect}
\newcommand{\subsct}[2][default]{\refstepcounter{subsection}
\addcontentsline{toc}{subsection}
{{\tocsection{\!\!}{\hspace{1.2em}\thesubsection}{\!\!\!\!#1\dotfill}}{}}
\nopagebreak\vspace{0.45\baselineskip} {\flushleft\bf
\thesubsection~\bf #1.~}
\\*[3mm]\noindent
\nopagebreak}
\newcommand{\sbsect}[1]{\vspace{0.1cm}\noindent
\textbf{#1.~}\vspace{0.1cm}}

\renewcommand{\subsubsection}{%
\secdef \subsubsect\sbsbsect}
\newcommand{\subsubsect}[2][default]{%
\refstepcounter{subsubsection}
\addcontentsline{toc}{subsubsection}{{\tocsection{\!\!}
{\hspace{3.05em}\thesubsubsection}{\!\!\!\!#1\dotfill}}{}}
\nopagebreak
\vspace{0.15\baselineskip} \nopagebreak {\flushleft\rmfamily
\itshape\thesubsubsection
\ \rmfamily #1\/.}\ }
\newcommand{\sbsbsect}[1]{\vspace{0.1cm}\noindent
\rmfamily \itshape
\arabic{section}.\arabic{subsection}.\arabic{subsubsection} \
\sffamily #1\/.\ }

\newcommand{\dist}{\operatorname{dist}}

\newcommand{\supp}{\operatorname{supp}}
\newcommand{\diam}{\operatorname{diam}}

\newcommand{\FF}{\mathcal F}

\newcommand{\MM}{\mathcal M}

\newcommand{\C}{\mathbf C}
\newcommand{\D}{\mathbb D}

\newcommand{\N}{\mathbb N}

\newcommand{\BbbP}{\mathbb P}

\newcommand{\R}{\mathbb R}

\newcommand{\Z}{\mathbb Z}

\definecolor{Blue}{rgb}{0,0,1}
\def\blue{\color{Blue}}

\def\myffrac#1#2 in #3{\raise 2.6pt\hbox{$#3 #1$}\mkern-1.5mu\raise 0.8pt\hbox{$#3/$}\mkern-1.1mu\lower 1.5pt\hbox{$#3 #2$}}

\newcommand{\joint}{\Large \texttt{P}}
\newcommand{\quenchedP}{P}
\newcommand{\annealedP}{\BbbP}
\newcommand{\quenchedE}{E}

\newcommand{\numsinks}{A}
\newcommand{\cubenz}[2]{Q_{#1}(#2)}
\newcommand{\ConnectExp}{\alpha}
\newcommand{\thesink}{{\mathcal C}}
\newcommand{\ProbBndinSink}{\xi}
\newcommand{\GoingTo}[2]{\Psi_{#1}(#2)}
\newcommand{\GoRen}[2]{\Psi_{#1}^{(#2)}}
\newcommand{\Cont}{\Phi}
\newcommand{\n}{{n^*}}
\newcommand{\nn}{{m^*}}

\newcommand{\Width}[2]{W_{#1}^{(#2)}}
\newcommand{\DistSink}[3]{\mbox{dist}_{#3}(#1,#2)}
\newcommand{\BernClust}{\mathcal D}
\newcommand{\BernSink}{\hat{\mathcal C}}
\newcommand{\ones}{{\underline{\bf 1}}}
\newcommand{\PercHarmConst}{C^\prime}
\newcommand{\connom}{\includegraphics[width=0.1in]{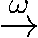}}
\newcommand{\connomp}{\includegraphics[width=0.1in]{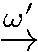}}
\newcommand{\conneib}{\includegraphics[width=0.08in]{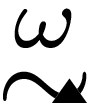}}
\newcommand{\desball}[2]{B^{\tiny \mbox{dis}}_{#1}(#2)}
\newcommand{\ball}[2]{B_{#1}(#2)}
\newcommand{\cball}[1]{B_{#1}}  
\newcommand{\disball}{B^{\tiny \mbox{dis}}}
\DeclareMathOperator*{\osc}{osc}

\newcommand{\pbm}{P_{\rm BM}}
\newcommand{\BndSecDer}[2]{M^{#1}_{#2}}
\newcommand{\CovMat}{{\Sigma}}
\newcommand{\QCovMat}[3]{M_{#2}^{(#3)}(#1)}

\newcommand{\domain}{\mathcal O}
\newcommand{\disdomain}[1]{\mathcal O^{{\tiny \mbox{dis}}}_{#1}}
\newcommand{\ExitProb}[1]{\chi_{#1}}
\newcommand{\Cube}[1]{Q_{#1}}
\newcommand{\Apbnd}{\xi}
\newcommand{\CubeColl}[1]{\mathcal Q_{#1}}
\newcommand{\Path}{\gamma}
\newcommand{\stair}[5]{
\draw[#5, #4] (o)--++(0,#1)--++(1,0) 
  \foreach \i in {1,...,#2}{%
  		--++({int(rnd*3)},0)--++(0,#3) coordinate (o)
		};
}
\newcommand{\OscEvent}[4]{H_{#1,#2}^{#3,#4}} 
\newcommand{\OscConst}{\upsilon}

\newcommand{\errQb}{{\tilde\epsilon}}
\newcommand{\BallGrowth}{M}
\newcommand{\BallScale}{K}
\newcommand{\disribTile}{{\mathcal D}}
\newcommand{\admis}{{\mathcal A}}
\newcommand{\BasicCoupling}[1]{{\tilde\mu^{(#1)}}}
\newcommand{\CompCoupling}[1]{{\mu^{(#1)}}}
\newcommand{\StopLevel}{\mathbb T}
\newcommand{\HitExp}{\rho}
\newcommand{\RegFlag}{{\mathcal K}}
\newcommand{\NonAdmisExp}{{\large \kappa}}
\newcommand{\infl}[2]{{\mathcal I}[{#1}\to{#2}]}
\newcommand{\infed}[1]{I({#1})}
\newcommand{\MeetProb}{\varrho}
\newcommand{\SizeOsc}{{\underline{\overline{\Psi}}}}
\newcommand{\LessThanTwo}{\tilde \alpha}

\newcommand{\hitmstarx}[2] {{\mathcal P}_{#1}^{#2}}

\newcommand{\PwrDwn}{{\Xi}}
\newcommand{\GoodEvents}[3]{{\mathcal G}_{#3}(#1,#2)}
\newcommand{\Mesh}{\nu}
\newcommand{\DistribGood}[5] {\tet(#1,#2,#3,#4,#5)}
\newcommand{\tet}{
\begin{picture}(10,10)
\put(10,10){\line(0,-1){10}}
\put(10,0){\line(-1,0){10}}
\put(0,0){\line(0,1){10}}
\put(10,10){\line(-1,0){5}}
\put(5,10){\line(0,-1){5}}
\end{picture}}

\newcommand{\rdconst}{\tilde H}
\newcommand{\omegaconst}{H}
\newcommand{\ConstClaimPart}{C_1}
\newcommand{\EpsComp}{D}

\title[Harnack inequality in balanced environments]
{\fontsize{14}{20}\selectfont An elliptic Harnack inequality for random walk in balanced environments}
\author[N.~Berger, M.~Cohen, J.-D.~Deuschel and X.~Guo]{Noam Berger\,$^{1,}$\,$^2$ \and Moran Cohen\,$^2$ \and Jean-Dominique Deuschel\,$^3$ \and Xiaoqin Guo\,$^4$}

\begin{document}

\thanks{\hglue-4.5mm\fontsize{9.6}{9.6}\selectfont\copyright\,2018 by N.~Berger, M.~Cohen, J.-D.~Deuschel and X.~Guo.
\vspace{2mm}}
\maketitle

\vspace{-5mm}
\centerline{\textit{$^1$The Technical University of Munich}}
\centerline{\textit{$^2$The Hebrew University of Jerusalem}}
\centerline{\textit{$^3$The Technical University of Berlin}}
\centerline{\textit{$^4$ University of Wisconsin--Madison}}

\begin{abstract}
We prove a Harnack inequality for the solutions of a difference equation with non-elliptic balanced i.i.d. coefficients. Along the way 
we prove a (weak) quantitative homogenisation result, which we believe is of interest too.
\end{abstract}

\section{Introduction}
\label{sec:intro}\noindent

This paper deals with a Random Walk in Random Environment (RWRE) on $\Z^d$  which is defined as follows: Let
$\MM^d$ denote the space of all probability measures on the nearest neighbors of the origin $\{\pm e_i\}_{i=1}^d$, with the $\sigma$-algebra which is inherited from the finite-dimensional space in which it is embedded, and let $\Omega=\left(\MM^d\right)^{\Z^d}$ with the product $\sigma$-algebra. An {\em environment} is a  point $\omega\in\Omega$, we denote by
$P$ the distribution of the environment on $\Omega$. For the purposes of this paper, we assume that $P$ is an i.i.d. measure, i.e.
for a given environment $\omega\in\Omega$, the {\em Random Walk} on $\omega$ is a time-homogenous
Markov chain jumping to the nearest neighbors with transition kernel
\[
\quenchedP_\omega\left(\left.X_{n+1}=z+e\right|X_n=z\right)=\omega(z,e)\ge 0, \ \ \ \ \ \ 
\sum_e\omega(z,e)=1.
\]
The {\em quenched law} $\quenchedP_\omega^z$ is defined to be the law on $\left(\Z^d\right)^\N$ induced by
the kernel
$\quenchedP_\omega$ and $\quenchedP_\omega^z(X_0=z)=1$.  We let $\joint^z=P\otimes \quenchedP_\omega^z$ be the joint law of the environment and the walk, and the {\em annealed} law is defined to be its marginal
\[
\annealedP^z=\int_{\Omega}P_\omega^zdP(\omega).
\]


A comprehensive account of the results and the remaining challenges in the understanding of RWRE can be found in Zeitouni's Saint Flour lecture notes \cite{ofernotes} and in the more recent survey paper by Drewitz and Ram\'irez \cite{dreram}.

\ignore{\red say what's elliptic. Mention that the perturbative is open in 2d}

In this paper we will focus on a special class of environments: the  balanced environment.
In particular, we solve the challenge of adapting the methods that were developed for the elliptic case in \cite{Lawler82} and \cite{GZ} to non-elliptic cases.

\begin{definition}\label{def:balanced}
An environment $\omega$ is said to be {\em balanced} if for every $z\in\Z^d$ and neighbor $e$ of the origin, $\omega(z,e)=\omega(z,-e)$.
\end{definition}

Of course we want to make sure that the walk really spans $\Z^d$:  
\begin{definition}\label{def:genddim}
An environment $\omega$ is said to be {\em genuinely $d$-dimensional} if for every neighbor $e$ of the origin, there exists $z\in\Z^d$ such that  $\omega(z,e)>0$.
\end{definition}

Throughout this paper we make the following assumption.
\begin{assumption}\label{ass:main}
$P$-almost surely, $\omega$ is balanced and genuinely $d$-dimensional.
\end{assumption}

Note that whenever the distribution is ergodic, the above assumption is equivalent with
$$P\big[\omega(z,e)=\omega(z,-e)\big]=1,\qquad\text{and}\qquad  P\big[\omega(z,e)>0\big]>0$$
for every $z\in\Z^d$ and a neighbor $e$ of the origin.

Note that unlike \cite{GZ} we do not allow holding times in our model. We do this for the sake of simplicity. Holding times in our case could be handled exactly as they are handled in \cite{GZ}.

\begin{example}\label{exam:Erwin's walk}
Take $P=\nu^{\Z^d}$ as above with
$$\nu\left[\omega(z,e_i)=\omega(z,-e_i)=\frac12, 
\omega(z,e_j)=\omega(z,-e_j)=0, \forall j\ne i\right]=\frac{1}{d}, \,\, i=1,...,d.$$
In this model, the environment chooses at random one of the $\pm e_i$ direction, see Figure \ref{fig:erwex}).

\begin{figure}[h]
\begin{center}
\includegraphics[width=2in]{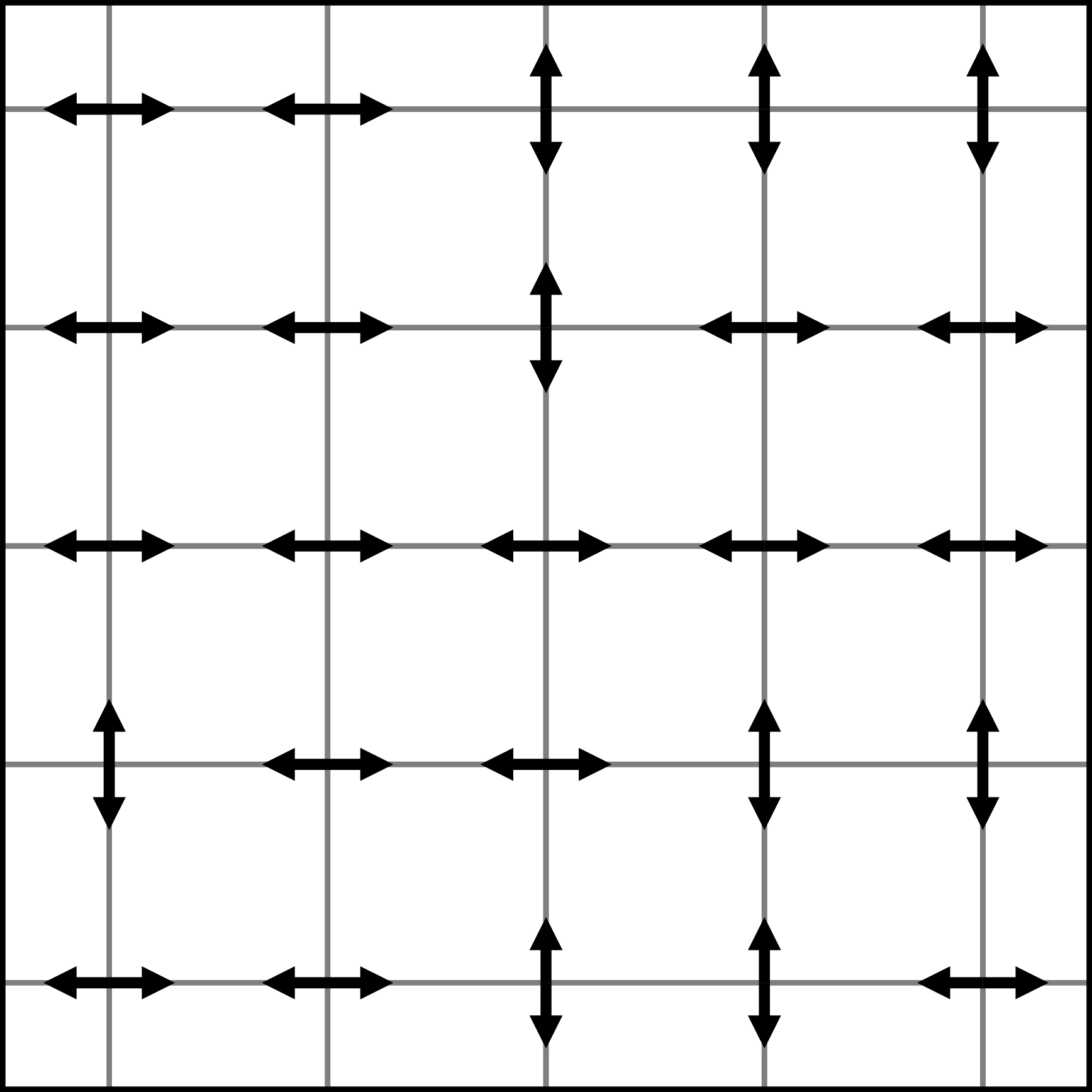}
\caption{\sl An illustration of Example \ref{exam:Erwin's walk} restricted to a small box.
}
\label{fig:erwex}
\end{center}
\end{figure}  
\end{example}

\subsection{Main results}

We have two main results in this paper. The first result is a sort of quantitative homogenization. In \cite{BD14} the following quenched invariance principle was proved.
Let
$$
\bar\omega_n = \tau_{-X_n}\omega,\quad n \in
\mathbb N
$$
where $\tau$ is the shift on $\Omega$, be the environment viewed from the point of view of the particle.

\begin{theorem}[\cite{BD14}]\label{thm:BD14} Assume that the environment is i.i.d., balanced and genuinely $d$-dimensional, then there exists a unique invariant distribution  $Q$ on $\Omega$
for the process $\{\bar\omega_n\}, n\in\mathbb N\}$ which is absolutely continuous with respect to $P$,
moreover the following quenched invariance principle holds:
for $P$ almost $\omega$,  the rescalled random walk $\{X^N(t)=N^{-1/2}X_{[Nt]}, t\ge 0\}$
converges weakly  under $P^0_\omega$ to  a Brownian motion  with deterministic non-degenerate diagonal covariance matrix $\Sigma$, 
 $\Sigma_{i,i}=2\Bbb E_{Q}[\omega(0,e_i)], i=1,...,d$. 
\end{theorem}

Then the content of Theorem \ref{thm:BD14} is that on  the large scale, the RWRE behaves like a Brownian Motion with 
covariance matrix $\Sigma$. The next theorem gives a quantitative bound on how much time it takes until this behavior is seen.

Namely,
we will provide a quantitative estimate (that holds with high probability) for the difference between the discrete-harmonic function and the corresponding {\it homogenized} function, where the two functions have the same (up to discretization) boundary
conditions. 

To be specific, for any discrete finite set $E\subset\Z^d$, we say that a function $u: \bar E\to\R$ is {\it $\omega$-harmonic} in $E$ if for any $x\in E$,
\[
L_\omega u(x):=\sum_{y:y\sim x}\omega(x,y)[u(y)-u(x)]=0.
\]
Let $\cball{R}=\{z\in\R^d:\|z\|_2<R\}$ be the ball of radius $R$ in $\R^d$. For $x\in\R^d$, let $\ball R x=x+\cball R$ and let $\desball{R}{x}=\ball{R}{x}\cap\Z^d$.

Let $\domain=\cball 1$ be the unit ball.
Let $\CovMat$ be the covariance matrix of the limiting Brownian Motion as in Theorem \ref{thm:BD14}. 
Let $F:\overline{\cball 1}\to\R$ be a function which is continuous in $\overline{\cball 1}$, smooth in $\cball 1$ and satisfies 
\[
\sum_{i,j=1}^d\CovMat_{ij}\partial_{ij}F(x)=0 \quad
 \mbox{ for all } x\in \cball 1.
\]
For given $R>0$, we denote by  
\[
\disdomain{R}:=\{x\in\disball_R: y\in \disball_R \mbox{ for all }y\sim x\}
\] 
the biggest subset of $\disball_R$ such that $\overline{\disdomain R}=\disball_R$.
For any $k>0$, we let 
\begin{equation}\label{eq:defIk}
I_k(\disdomain R):=\{x\in\disdomain R:\dist(x, \partial\disdomain R)>k\}
\end{equation}
denote the subset of $\disdomain R$ that has distance more than $k$ away from $\partial\disdomain R$. 

We define the function $F_R: \overline{\disdomain{R}}\to\R$ by 
\[
F_R(x)=F(x/R) \quad \mbox{ for every $x\in\overline{\disdomain{R}}$},
\] so that $F_R$ is the function obtained by ``stretching"  the domain of $F$.
Let $G=G_{R,\omega}: \overline{\disdomain{R}}\to\R$ be the solution of the Dirichlet problem
\[
\left\{
\begin{array}{lr}
L_\omega G=0 &\mbox{ on }\disdomain{R}\\
G=F_R &\mbox{ on }\partial\disdomain{R}.
\end{array}
\right.
\]
In other words, $G$ is the $\omega$-harmonic function on $\overline{\disdomain{R}}$ whose boundary data on $\partial\disdomain{R}$ agrees with that of $F_R$. 

For $i\in\N$, let $\BndSecDer{i}{F}$ denote the supremum (of the absolute values) of all $i$-th order partial derivatives of $F$ on $\cball 1$.

we can thus state the following quantitative estimate.

\begin{theorem}\label{thm:quan_hom}
For any $\epsilon>0$, there exists $n_1=n_1(\epsilon,P)>\epsilon^{-2}$ and $\delta=\delta(P)>0$ such that for any $R>n_1$, 
with probability greater than or equal to $1-C\exp(-R^\delta)$,
\begin{equation*}
\max_{x\in \disdomain{R}}|F_R(x)-G_{R,\omega}(x)|\leq \epsilon(\BndSecDer{2}{F}+\BndSecDer{3}{F}).
\end{equation*}
\end{theorem}

As a consequence of the above Theorem we get the following result on the exit distribution from a large ball for the random walk:
Let $A\subseteq\partial B_1$ be open in the
relative topology of $\partial B_1$. We write $\beta(A)$ for the probability that a Brownian Motion starting at the origin with covariance matrix $\Sigma$
 leaves the ball $B_1$ through $A$.
Fix $\epsilon>0$ and $R$, and we say the the environment $\omega$ is $(A,R,\epsilon)$-good if the probability that the RWRE in environment $\omega$ leaves $B_R$ through $RA$ is within $\epsilon$ from $\beta(A)$.

\begin{corollary}\label{cor:quan_hom} Assume that the environment is i.i.d., balanced and genuinely $d$-dimensional. There exists $\alpha>0$ such that for every $A$ and $\epsilon$ there are finite and positive constants $C_1$ and $C_2$ such that the probability that $\omega$ is  {$(A,R,\epsilon)$}-good is greater than or equal to $1-C_1\exp(-C_2R^\alpha)$.
\end{corollary}

Our second main result is a Harnack inequality for functions that are $\omega$-harmonic. 
Then we prove the following which is our main result
\begin{theorem}\label{thm:harnack}
Under Assumption \ref{ass:main}, there exists a constant $C$ and a random $R_0$ satisfying $P(R_0<\infty)=1$ such that for every $R>R_0$ and every non-negative $\omega$-harmonic function $f:B_{2R}\to\R$, we have
\[
\sup\{f(x):x\in B_R\}\leq C \inf\{f(x):x\in B_R\}.
\]
Furthermore, there exist $\beta>0$ and $\gamma>0$ such that for every $M$,
\[
P(R_0>M) < \exp(-\gamma M^\beta).
\]
In addition, the constant $C$ can be taken arbitrarily close to the constant in the classical Harnack inequality in $\R^d$ for the corresponding
Brownian motion with covariance $\Sigma$.
\end{theorem}

From Theorem \ref{thm:harnack} we get the following weak Liouville type estimate.
\begin{corollary}\label{cor:liouville}
Let $\delta < 1$. For $P$-almost $\omega$ and every $\omega$-harmonic function $f : \Z^d \to \R$, if
\[
\sup_{0 < |x|} \frac{|f(x)|}{|x|^\delta} <\infty
\]
then $f$ is a constant function.
\end{corollary}

We note that Harnack inequalities for balanced environments in the elliptic setting have been proven before, see \cite{KuoTru, ofernotes, lawlerharnack}, and for other non-elliptic cases,
see \cite{Barlow}. However, we believe that our result is the first Harnack inequality in the context of RWRE which is valid in a case which is non-reversible and non-elliptic. It
is also to the best of our knowledge the first Harnack inequality in the context of RWRE where the optimal Harnack constant is established.

\subsection{Structure of the paper}
Many of our arguments are based on results established in \cite{BD14}.
In Section \ref{sec:inpbd} we collect those results so that we can use them later in the paper.
In Section \ref{sec:quan_hom} we prove Theorem \ref{thm:quan_hom}.
In Section \ref{sec:perc} we discuss the connectivity structure of the balanced directed percolation. While proving some results of independent interest, the main goal of this section is to provide
percolation estimates necessary for the proof of Theorem \ref{thm:harnack}.
Then, finally, in Sections \ref{sec:osc} and \ref{sec:harnack} we collect results from 
previous sections and prove Theorem \ref{thm:harnack}. In particular, the proof combines ideas for Harnack inequalities that we learned from \cite{KuoTru} and \cite{Barlow}, some of which go back to \cite{FS86}.

\subsection{Input from \cite{BD14}}\label{sec:inpbd}
In this section we review some useful definitions and results from \cite{BD14}.

We first define the rescaled walk, which is a useful notion in the study of non-elliptic balanced RWRE, and recall  some basic facts about it.

Let $\{X_n\}_{n=0}^\infty$ be a nearest neighbor walk in $\Z^d$, i.e. a sequence in $\Z^d$ such that $\|X_{n+1}-X_n\|_1=1$ for every $n$. Let $\alpha_n,n\geq 1$ be the coordinate that changes between $X_{n-1}$ and $X_{n}$, i.e. $\alpha(n)=i$ whenever $X_{n}-X_{n-1}=e_i$ or $X_{n}-X_{n-1}=-e_i$.

The following is  \cite[Definition 3]{BD14}.
\begin{definition}\label{def:scrw}
The stopping times $T_k,k\geq 0$ are defined as follows: $T_0=0$. Then
\[
T_{k+1}=\min\left\{
t>T_k:\{\alpha(T_k+1),\ldots,\alpha(t)\}=\{1,\ldots,d\}
\right\}\leq\infty.
\]
We then define the {\em rescaled random walk} to be the sequence (no longer a nearest neighbor walk)
$Y_n=X_{T_n}$. $(Y_n)$ is defined as long as $T_n$ is finite.
\end{definition}

The following estimates both annealed and quenched have been derived in  \cite{BD14}, cf. Lemma 2.1, 2.2, 2.3 and 2.4:
\begin{lemma}\label{lem:well_defined}
$\annealedP$-almost surely, $T_k<\infty$ for every $k$.
There exists a constant $C$ such that for every $n$,
\[
\annealedP(T_1>n)<e^{-Cn^{\frac 13}}.
\]
and
\[
P\left(
\omega:\quenchedE_\omega(T_1)>k
\right)
\leq e^{-Ck^{\frac 13}}.
\]
Moreover for every $0<p<\infty$, 
\[
E\left[
\quenchedE_\omega(T_1^p)
\right]<\infty.
\]
\end{lemma}

The maximum principle, originally due to Alexandrov-Bakelman-Pucci in the continuum and adapted to
 balanced random walks by Kuo and Trudinger, \cite{KuoTru}  is one key analytical tool in the proof 
 of the existence of the stationary measure for the walk $Q$. In our non-elliptic setting it
 can be restated in terms of the rescaled process as follows, cf Theorem 3.1 of \cite{BD14}:
 
 For $N\in\N$ and $k=k(N)\in (0,N)\cap\Z$, let $T_1^{(N)}=T_1^{(N,k)}=\min\big(T_1, k\big)$.
Let $h:\Z^d\to\R$ be a real valued function, and for every $z\in\Z^d$, let
$L^{(N)}_\omega h(z):=h(z)-\quenchedE_\omega^z[h(X_{T_1^{(N)}})]$.

Let $Q\subseteq \Z^d$ be finite and connected, and let
$\partial^{(k)}Q=\{z\in \Z^d-Q: \exists_{x\in Q}\|z-x\|_\infty < k\}$.

We say that a 
 point $z\in Q$ is {\em exposed} if there exists $\beta=\beta(z,h)\in\R^d$ such that 
$h(z)-\langle\beta,z\rangle\geq h(x)-\langle\beta,x\rangle$
for every $x\in Q\cup\partial^{(k)}Q$. We let $D_h$ be the set of exposed points.
Further, we define the {\em angle of vision} $I_h(z)$ as follows:
\begin{equation}\label{eq:defIz}
I_h(z)=\left\{
\beta\in\R^d: \forall_{x\in Q\cup\partial^{(k)}Q}  h(x)\leq h(z)+\langle\beta,x-z\rangle
\right\}.
\end{equation}
This is the set of hyperplanes that touch the graph of $h$ at $(z,h(z))$ and are above the graph of  $h$ all over
$Q\cup\partial^{(k)}Q$. A point $z$ is exposed if and only if $I_h(z)$ is not empty.

%

\begin{theorem}[Maximum Principle, \cite{BD14} Theorem 3.1]\label{thm:max_princ}
There exists $N_0$ such that for every $N>N_0$ and every $0<k<N$, every balanced environment 
$\omega$ and every $Q$ of diameter $N$,
if for every $z\in Q$
\begin{equation}\label{eq:condformax}
\quenchedP_\omega^z\big(T_1>(\log N)^\kappa\big)<e^{-(\log N)^2}
\quenchedE_\omega^z\big(T_1\cdot{\bf 1}_{T_1>k}\big)<e^{-(\log N)^2}
\quenchedP_\omega^z\big(T_1>k\big)<e^{-(\log N)^3}
\end{equation}
then
\begin{equation}\label{eq:max_princ}
\max_{z\in Q}h(z)-\max_{z\in \partial^{(k)}Q}h(z)
\leq 6 N \left(
\sum_{z\in Q} 
{\bf 1}_{z\in D_h}
|L^{(N)}_\omega h(z)|^d
\right)^{\frac 1d}
\end{equation}
\end{theorem}


We now turn to some definitions and results pertaining to Percolation.

\begin{definition}\label{def:connected}
For $\omega\in\Omega$ and $x,y\in\Z^d$, we denote by  
$
x\connom y
$
the occurrence
\[
\quenchedP_\omega^x(\exists_n X_n=y)>0.
\]
We say that a set $A\subseteq\Z^d$ is {\em strongly connected} w.r.t. $\omega$ if for every $x$ and $y$ in $A$, 
$
x\connom y.
$
A set $A\subseteq\Z^d$ is called a {\em sink} w.r.t. $\omega$ if it is strongly connected and 
$x \hspace{-0.15cm} \not\connom y$ for every $x\in A$ and $y\notin A$.

For every two neighbors $x$ and $y$, we draw
a directed edge from $x$ to $y$, denoted $x\conneib y$, whenever 
$\omega(x,y-x)>0$.

\end{definition}
For a measure $Q$ which is invariant w.r.t. the point of view of the particle and is absolutely continuous w.r.t. $P$, we define 
\[\supp Q=\{\omega:\frac{dQ}{dP}(\omega)>0\},\]
where the derivative is the Radon-Nykodim derivative. This is well define up to a set of $P$-measure zero.

We define $\supp_\omega Q = \{z\in\Z^d:\tau_{-z}(\omega)\in \supp Q\}$.
In view of Corollary 4.12 and Lemma 5.6 of \cite{BD14} we have the following lower bound on the density of a sink:
\begin{lemma}\label{lem:posdens}
\begin{enumerate}
\item\label{item:posdens}
There exists $\Phi>0$ such that for $P$-almost every $\omega$, every sink has lower density at least $\Phi$.
\item\label{item:finmany}
For every ergodic $Q$ which is invariant w.r.t. the point of view of the particle and is absolutely continuous w.r.t. $P$, $P$-a.s. there are only finitely many sinks contained in $\supp_\omega Q$.
\item\label{item:allpnt} $P$-a.s., every point in $\supp_\omega Q$ is contained in a sink.
\end{enumerate}
\end{lemma}
In other words, the lemma says that a.s. $\supp_\omega Q$ is a finite union of sinks, each of which has lower density at least $\Phi$.


As announced in Remark 3 of \cite{BD14} let us now state 
\begin{proposition}\label{prop:unique}
There exists a unique sink.
\end{proposition}

\begin{proof}
We use an adaptation of the easy part of the percolation argument of Burton and Keane \cite{BuKe},
easy since we already know that there are finitely many sinks.
Even though the finite energy condition is not satisfied, a very similar yet slightly weaker condition holds.
Let $S_1(\omega), S_2(\omega), ..., S_k(\omega)$ be distinct sinks, $k\ge 2$.
Note that by ergodic theorem the number of sinks is $P$ a.s. constant.
Define 
\[
\dist(P):=\min\left( |z-w|:z \mbox{ and } w \mbox{ are in two distinct sinks} \right).
\]
Note that due to shift invariance and ergodicity $\dist(P)$ is a $P$-almost sure constant, and therefore we treat it as a natural number rather than as a random variable.
Let $z$ and $w$ be two points such that $|z-w|=\dist(P)$, and such that the event $U=U(z,w)=\{z \mbox{ and } w \mbox{ are in two distinct sinks}\}$ has a positive $P$ probability. Let $i$ be a direction s.t. $\langle e_i,z-w\rangle\neq 0$.
 Let $R$ be the following measure on $\Omega\times\Omega$: we sample $\omega$ and $\omega'$. for all $x\neq z$, we take $\omega(x)=\omega'(x)$ to be sampled i.i.d. according to $\nu$. We then take $\omega(z)\sim (\nu|\omega(e_i)=0)$ and $\omega'(z)\sim (\nu|\omega(e_i)\neq  0)$. Again, everything is independent. Let $P_1$ be the distribution of $\omega$ and $P_2$ be the distribution of $\omega'$.
 Note that $P_1$ and $P_2$ are both absolutely continuous w.r.t. $P$, and that $P_1(U)>0$ and $P_2(U)=0$.
 We now condition on the (positive probability) event $U(\omega) \setminus U(\omega^\prime)$.

Call $S_1$ the sink containing $z$ in $\omega$, call $S_2$ the sink containing $w$ in $\omega$, and $S_3, \ldots, S_k$ all the other sinks in $\omega$.
For $j=2, \ldots, k$, the environments $\omega$ and $\omega^\prime$ agree on $S_j$, so $S_2, \ldots, S_k$ are all sinks in $\omega^\prime$.

Assume w.l.o.g. that $\langle e_i,z-w\rangle > 0$, and let $z^\prime = z + e_i$. Then $z^\prime$ is in no sink in $\omega^\prime$ b/c it is too close to $S_2$.
Note that $x \connomp z^\prime$ for every $x \in S_1$. Thus no point in $S_1$ is in a sink in $\omega^\prime$. Further, no point outside of $S=\cup_{j=1}^k S_j$ is in a sink in $\omega^\prime$.
Indeed, if $y \in \Z^d \setminus S$ is in a sink $W$, then $W$ cannot intersect $S$, and thus $W$ is a sink in $\omega$ other than $S_1, \ldots, S_k$. However, there is no such sink, thus $W$ does not
exist.

Thus there are only $k-1$ sinks in $\omega^\prime$, in contradiction to $P_2$ being absolutely continuous with respect to $P$.
 
\end{proof}

The last result that we need is the following lemma, which follows immediately from \cite[Proposition 5.9]{BD14} and Proposition \ref{prop:unique}
\begin{lemma}\label{lem:BD1459}
Let $S$ be the (a.s. unique) sink. Then for $P$-a.e. $\omega$ and every $z\in\Z^d$,
\[
\quenchedP_\omega^z\left(
\exists_N \mbox{ s.t. } \forall_{n>N} X_n\in S
\right)=1.
\]
\end{lemma}


\section{Quantitative estimates for the invariance principle}\label{sec:quan_hom}

In this section we prove a quantitative homogenization bound. Namely,
we will provide a quantitative estimate (that holds with high probability) for the difference between the discrete-harmonic function and the corresponding {\it homogenized} function, where the two functions have the same (up to discretization) boundary
conditions. 

To be specific, for any discrete finite set $E\subset\Z^d$, we say that a function $u: \bar E\to\R$ is {\it $\omega$-harmonic} in $E$ if for any $x\in E$,
\[
L_\omega u(x):=\sum_{y:y\sim x}\omega(x,y)[u(y)-u(x)]=0.
\]
Let $\cball{R}=\{z\in\R^d:\|z\|_2<R\}$ be the ball of radius $R$ in $\R^d$. For $x\in\R^d$, let $\ball R x=x+\cball R$ and let $\desball{R}{x}=\ball{R}{x}\cap\Z^d$.

Let $\domain=\cball 1$ be the unit ball.
Let $\CovMat$ be the covariance matrix of the limiting Brownian Motion as in Theorem \ref{thm:BD14}. 
Let $F:\overline{\cball 1}\to\R$ be a function which is continuous in $\overline{\cball 1}$, smooth in $\cball 1$ and satisfies 
\[
\sum_{i,j=1}^d\CovMat_{ij}\partial_{ij}F(x)=0 \quad
 \mbox{ for all } x\in \cball 1.
\]
For given $R>0$, we denote by  
\[
\disdomain{R}:=\{x\in\disball_R: y\in \disball_R \mbox{ for all }y\sim x\}
\] 
the biggest subset of $\disball_R$ such that $\overline{\disdomain R}=\disball_R$.
For any $k>0$, we let 
\begin{equation}\label{eq:defIk}
I_k(\disdomain R):=\{x\in\disdomain R:\dist(x, \partial\disdomain R)>k\}
\end{equation}
denote the subset of $\disdomain R$ that has distance more than $k$ away from $\partial\disdomain R$. 

We define the function $F_R: \overline{\disdomain{R}}\to\R$ by 
\[
F_R(x)=F(x/R) \quad \mbox{ for every $x\in\overline{\disdomain{R}}$},
\] so that $F_R$ is the function obtained by ``stretching"  the domain of $F$.
Let $G=G_{R,\omega}: \overline{\disdomain{R}}\to\R$ be the solution of the Dirichlet problem
\[
\left\{
\begin{array}{lr}
L_\omega G=0 &\mbox{ on }\disdomain{R}\\
G=F_R &\mbox{ on }\partial\disdomain{R}.
\end{array}
\right.
\]
In other words, $G$ is the $\omega$-harmonic function on $\overline{\disdomain{R}}$ whose boundary data on $\partial\disdomain{R}$ agrees with that of $F_R$. 

For $i\in\N$, let $\BndSecDer{i}{F}$ denote the supremum (of the absolute values) of all $i$-th order partial derivatives of $F$ on $\cball 1$.

Our goal in this section is to obtain Theorem \ref{thm:quan_hom}, namely that
%
For any $\epsilon>0$, there exists $n_1=n_1(\epsilon,P)>\epsilon^{-2}$ and $\delta=\delta(P)>0$ such that for any $R>n_1$, 
with probability greater than or equal to $1-C\exp(-R^\delta)$,
\begin{equation}\label{eq:maxH}
\max_{x\in \disdomain{R}}|F_R(x)-G_{R,\omega}(x)|\leq \epsilon(\BndSecDer{2}{F}+\BndSecDer{3}{F}).
\end{equation}


%

First, note that the ``stretched" version $F_R$ of the function $F$ is very ``flat" when $R$ is large. Indeed, for $x,y \in\overline{\disdomain{R}}\subset R\domain$, by Taylor expansion,
 \begin{equation}\label{eq:Taylor_F}
 F_R(y)-F_R(x)=
 \frac{1}{R}\langle \triangledown F(\tfrac{x}{R}),y-x\rangle + \tfrac{1}{2R^2}(y-x)^t D_2 F(\tfrac{x}{R})(y-x) +\rho_y\|y-x\|^3,
 \end{equation}
where the error term $\rho_y$ is bounded by $R^{-3}\BndSecDer{3}{F}$.
 Hence, we conclude that 
 {
\begin{equation}\label{eq:LF}
|L_\omega F_R(x)|
\le 
\frac{\BndSecDer{2}{F}}{R^2}+\frac{2d\BndSecDer{3}{F}}{R^3}
\le \frac{C_F}{R^2} 
\qquad \mbox{ for all }x\in\disdomain{R}, 
\end{equation}
}
where $C_F=
C(\BndSecDer{2}{F} +2d\BndSecDer{3}{F})$. Note that the constant $C$ may differ from line to line and so may $C_F$.

\bigskip
Our next observation is that since (by the quenched CLT) the diffusion matrix of  $X_n/\sqrt n$ converges to $\CovMat$, the function $F_R$ should be ``approximately $\omega$-harmonic" in a sense that will be made precise in \eqref{eq: LF_n0} below. Indeed, for $x\in\Z^d, n\in\N$, let $\QCovMat{x}{\omega}{n}$ be the $d\times d$ covariance matrix with entries
\[
(\QCovMat{x}{\omega}{n})_{ij}:=E_\omega^x[(X_n(i)-x)(X_n(j)-x)]/n, \quad 1\le i,j\le n,
\]
where $X_n(i)$ denotes the $i$-th coordinate of $X_n$.
For any fixed $\epsilon>0$ and $\alpha>0$, by 
Theorem \ref{thm:BD14} 
 there exists $n_0=n_0(\epsilon,\alpha, P)$ such that for any $n\ge n_0$,
\begin{equation}\label{eq:bdest}
P\left[\|\QCovMat{x}{\omega}{n}-\CovMat\|_1<\epsilon\right]>1-\alpha.
\end{equation}
Moreover, for any $x\in I_{n_0}(\disdomain R)$,  
by \eqref{eq:Taylor_F} we have
\begin{align*}
& E_\omega^x[F_R(X_{n_0})-F_R(x)]\\
&=
\frac{1}{R^2}E_\omega^x[(X_{n_0}-x)^t D_2 F(\tfrac{x}{R})(X_{n_0}-x)]+ E_\omega^x[\rho_{X_{n_0}}\|X_{n_0}-x\|^3]\\
&=
\frac{n_0}{R^2}\sum_{i,j=1}^d \partial_{ij} F(\tfrac{x}{R})\QCovMat{x}{\omega}{n_0}_{ij}+E_\omega^x[\rho_{X_{n_0}}\|X_{n_0}-x\|^3].
\end{align*}
Here $D_2F(x)$ denotes the Hessian matrix of $F$ at $x$.
Under the event
\begin{equation}\label{eq:def-An0}
A_{n_0}(x)=\{\omega: \|\QCovMat{x}{\omega}{n_0}-\CovMat\|_1<\epsilon\},
\end{equation}
recalling that
$\sum_{i,j=1}^d \partial_{ij} F(\tfrac{x}{R}) \CovMat_{ij}=0$, 
we see that 
\[
\left|
\sum_{i,j=1}^d \partial_{ij}\ F(\tfrac{x}{R})\QCovMat{x}{\omega}{n_0}_{ij}
\right|
\leq
\epsilon d^2 \BndSecDer{2}{F}.
\]
Hence for $x\in I_{n_0}(\disdomain R)$,  when $A_{n_0}(x)$ occurs we  obtain
\begin{equation}\label{eq: LF_n0}
\left|
\quenchedE_\omega^x[F_R(X_{n_0})] - F_R(x)
\right|
\leq 
\epsilon n_0d^2R^{-2}\BndSecDer{2}{F} + C\BndSecDer{3}{F}{n_0}^{\frac 32}R^{-3} 
\leq 
\epsilon n_0R^{-2}C_F
\end{equation}
for $R>\tfrac{\sqrt{n_0}}{\epsilon}\vee n_0$.

\bigskip

\ignore{
Now we are ready to use a discrete version of the classical maximum principle of Alexandrov, Bakhelman and Pucci 
({\red Xiaoqin - references}) to prove Theorem~\ref{thm:quan_hom}. Of the many versions of this theorem, we use the version \cite[Theorem 3.1]{BD14} which was tailored especially for our non-elliptic environment. 

For any finite set $Q\subset\Z^d$, let $\partial^{(k)}Q=\{x\in\Z^d\setminus Q: \dist(x, Q)\le k\}$. Note that $\partial^{(1)}Q=\partial Q$.
{\red XG: Need to copy the definitions of $T_1, L_\omega^{(N)}, \partial^{(k)}$..}
\begin{lemma}[Theorem 3.1 of \cite{BD14}]\label{lem:max_princ}
There exists $N_0$ such that for every $N>N_0$ and every $0<k<N$, every balanced environment 
$\omega$ and every $Q\subset\Z^d$ of diameter $N$,
if 
\begin{equation}\label{eq:condformax}
\quenchedP_\omega^z\big(T_1>k\big)<e^{-(\log N)^3}
\qquad\mbox{ for all } z\in Q,
\end{equation}
then for any function $u: Q\cup\partial^{(k)}Q\to\R$,
\begin{equation}\label{eq:max_princ}
\max_{z\in Q}u(z)
\leq 
\max_{z\in \partial^{(k)}Q}u(z)+
6 N \left(
\sum_{z\in D_h} {\blue \big (-L^{(N)}_\omega u(z)\big)_+}^d
\right)^{\frac 1d}
\end{equation}
where 
\[
D_u=D_u(k,Q)=\left\{x\in Q:
\exists{\beta\in\R^d} \mbox{ such that }
u(x)-\langle\beta,x\rangle=\max_{z\in Q\cup\partial^{(k)}Q} 
u(z)-\langle\beta,z\rangle
\right\}
\]
is the upper contact set of the function $u$.
\end{lemma}
}

Recall the definition of $n_0$ and $A_{n_0}$ in \eqref{eq:bdest} and \eqref{eq:def-An0}. We claim that taking $k=\sqrt R$,  \eqref{eq:condformax} also happens with high probability. For fixed $\epsilon,\alpha>0$, we define the events (Here for any discrete set $Q$, $|Q|$ denotes the cardinality of $Q$.)
\begin{align*}
A^{(1)}_R&=
\left\{\omega:
\frac{\sum_{x\in\disball_R} {\bf 1}_{\omega\in A_{n_0}(x)}}{|\disball_R|} > 1-2\alpha
\right\},\\
A_R^{(2)}&=\left\{\omega: \quenchedP_\omega^x(T_1>R^{1/2})<e^{-(\log R)^3} \mbox{ for all }x\in\disdomain{R}\right\},
\end{align*}
%
 where the stopping time $T_1$ in the definition of $A_R^{(2)}$
is as in Definition \ref{def:scrw}.

\begin{lemma}\label{lem:events_likely}
Let $\epsilon\in(0,1),\alpha>0, n_0=n_0(\epsilon,\alpha,P)$ be the same as in \eqref{eq:bdest}. 
There exist $C=C(n_0,P)$ and $c=c(n_0,P)$ such that
$P(A^{(1)}_R\cap A^{(2)}_R)>1-Ce^{-cR^{1/7}}$.
\end{lemma}

We postpone the proof of Lemma \ref{lem:events_likely} until after the proof of Theorem \ref{thm:quan_hom}.

Throughout this section we always take $k=k(R):=\sqrt R$.

To prove Theorem~\ref{thm:quan_hom}, we consider the error
\[
H(x)=H_{R,\omega}(x):=F_R(x)-G_R(x),   \qquad x\in\overline{\disdomain{R}}.
\] 
Then  $H$  is the solution of the discrete Dirichlet problem
\[
\left\{
\begin{array}{lr}
L_\omega H=L_\omega F_R  &\mbox{ on }\disdomain{R}\\
H=0, &\mbox{ on }\partial\disdomain{R}.
\end{array}
\right.
\]
However, $H$ is not defined on $\partial^{(k)}\disdomain R$. To apply Theorem~\ref{thm:max_princ}, we define an auxiliary function  $H':\overline{\disball_{R+k}}\to\R$ to be the solution of the Dirichlet problem $H'|_{\partial\disball_{R+k}}=0$ and
\[
L_\omega H'=\left\{
\begin{array}{lr}
{
L_\omega F_R}
&\mbox{ on }\disdomain R\\
0&\mbox{ on }\disball_{R+k}\setminus \disdomain R.
\end{array}
\right.
\]
Notice that by the definitions of $H$ and $H'$,
\begin{equation}\label{eq:H'}
H'(x)= E_\omega^x
\left[\sum_{i=0}^{\tau_{R+k}-1}-L_\omega F_R(X_i){\bf 1}_{X_i\in \disdomain R}\right]
\qquad\mbox{ for all }x\in \overline{\disball_{R+k}}, \mbox{ and}
\end{equation}
\[
H(x)=E_\omega^x\left[
\sum_{i=0}^{\tau(\disdomain R)-1} -L_\omega F_R
(X_i)
\right]
\qquad\mbox{ for all }x\in \overline{\disdomain R},
\]
where $\tau(Q):=\inf\{n\ge 0: X_n\notin Q\}$ denotes the exit time from $Q\subset\Z^d$ and $\tau_R:=\tau(\disball_R)$.
Hence  by \eqref{eq:LF} and \eqref{eq:H'},  for $k=\sqrt R$,
{
\begin{equation}\label{eq:bdry_H'}
\max_{\partial\disdomain R}|H'|
\le 
\max_{x\in\partial\disdomain R} C_F E_\omega^x[\tau_{R+k}]/R^2\le C_F k/R=C_F/\sqrt R
\end{equation}
and so 
\[
\max_{\disdomain R}|H-H'|\le \max_{\partial \disdomain R}|H-H'|
=\max_{\partial \disdomain R}|H'|
\le C_F/\sqrt R.
\]
}

Now for any fixed small constant $\epsilon>0$, we set 
\[
\gamma:=3\epsilon \BndSecDer{2}{F} R^{-2}
\]
and define $h:\overline{\disball_{R+k}}\to\R$ as
\[
h(x)=H'(x)+\gamma\|x\|_2^2 \qquad \forall x\in\overline{\disball_{R+k}}.
\]
Our first goal is to use Theorem \ref{thm:max_princ} to estimate $\max_{\disdomain{R}}h$. We do so by showing that most of the points in
$\disball_{R+k}$ are outside of $D_h$ (Lemma \ref{lem:point_not_exposed}), and then controlling the $\omega$-Laplacian in the remaining points (Lemma \ref{lem:cont_T_1} below).

\begin{lemma}
\label{lem:point_not_exposed}
Let $\epsilon,\alpha>0$ be any fixed constants and let $n_0=n_0(\epsilon,\alpha, P)$ be as in \eqref{eq:bdest}. Let $R>\tfrac{\sqrt{n_0}}{\epsilon}\vee n_0$. 
For any $x\in I_{n_0}(\disdomain R)$, if $\omega\in A_{n_0}(x)$, then $x\notin D_h$.
\end{lemma}
From Lemma \ref{lem:point_not_exposed} we see that such $x$ is not counted in the sum in \eqref{eq:max_princ}. 

\begin{proof}[Proof of Lemma~\ref{lem:point_not_exposed}]
When $R>\tfrac{\sqrt{n_0}}{\epsilon}\vee n_0$, $x\in I_{n_0}(\disdomain R)$ and $\omega\in A_{n_0}(x)$, by \eqref{eq: LF_n0} and our definition of $\gamma$,
\begin{align*}
E_\omega^x[h(X_{n_0})-h(x)]
&=
E_\omega^x[H'(X_{n_0})-H'(x)]+\gamma E_\omega^x[\|X_{n_0}\|^2-\|x\|^2]\\
&= E_\omega^x[F_R(X_{n_0})-F_R(x)]+\gamma E_\omega^x[\|X_{n_0}\|^2-\|x\|^2]\\
&\ge 
-2\epsilon n_0R^{-2}\BndSecDer{2}{F}+\gamma n_0>0.
\end{align*}

Thus $\quenchedE_\omega^x[h(X_{n_0})]>h(x)$. In particular, since $\quenchedE_\omega^x[X_{n_0}]=x$, for every $\alpha\in\R^d$ we
get $\quenchedE_\omega^x[h(X_{n_0})+\langle\alpha,X_{n_0}-x\rangle]>h(x)$. So for every $\alpha\in\R^d$ there exists $y$ in the support of $X_{n_0}$ with 
$h(y)+\langle\alpha,y-x\rangle)>h(x)$, which implies $x\notin D_h$.
\end{proof}

Our next step is to control the $\omega$-Laplacian of the function $h$. By definition, on $B_{R+k}$,
\[
|L_\omega h|
\stackrel{\eqref{eq:LF}}{\le} 
|\gamma-{\bf 1}_{\disdomain R}C_F/R^2|
\leq C_F/R^2.
\] Hence
\begin{equation}\label{eq:omega_laplacian_bnd}
|L^{(R)}_\omega h(x)| \leq C_F\quenchedE_\omega^x[T_1]/R^2.
\end{equation}
For $p>0$ and $K>0$, we define the event
\[
A_R^{(3)}(p,K):=\left\{\omega: 
\frac{1}{|\disball_R|}\sum_{x\in\disball_
R}\left(\quenchedE_\omega^x[T_1]\right)^p
\le  K
\right\}.
\]

\begin{lemma}\label{lem:cont_T_1}
Let $p>d$. There exist positive constants $\delta$ and $C$ depending only on $p$ and the environment measure $P$ such that 
\[
P\left(A_R^{(3)}(p,C)\right)
>1- Ce^{-R^\delta}.
\]
\end{lemma}

We postpone the proof of Lemma \ref{lem:cont_T_1} until after the proof of Theorem \ref{thm:quan_hom}.

We are now ready to bound the function $H$ on $\disdomain{R}$.  

\begin{proof}[Proof of Theorem~\ref{thm:quan_hom}]
Let $\alpha=\alpha(\epsilon)>0$ be a small constant to be determined later.
Let $n_0=n_0(\epsilon,\alpha,P)$ be as  in \eqref{eq:bdest} and $R>\tfrac{\sqrt{n_0}}{\epsilon}\vee n_0$.
We only need to prove \eqref{eq:maxH} for 
\[
\omega\in A^{(1)}_R\cap A^{(2)}_R\cap A^{(3)}_R.
\]
First we estimate $\max_{x\in\disball_R}h(x)$. 
By Lemma~\ref{lem:point_not_exposed}, 
${\bf 1}_{x\in D_h}L^{(R)}_\omega h(x)\neq 0$ only if $\omega\notin A_{n_0}(x)$ or $x\in J_k:=\disball_R\setminus\disball_{R-k}$.
Note that $|J_k|\le CkR^{d-1}$.

By Theorem \ref{thm:max_princ}, on the event $A^{(2)}_R$,
\begin{eqnarray}\label{eq:from_mp}
\max_{x\in {\disball_R}}h(x)&-&\max_{x\in \partial^{(k)}{\disball_R}}h(x)
\leq 6 R \left(
\sum_{x\in {\disball_R}} {\bf 1}_{x\in D_h}(-L^{(R)}_\omega h(x))_+^d
\right)^{\frac 1d} \\
\nonumber
&\leq&
CR^2 \left(\frac{1}{|\disball_R|}
\sum_{x\in\disball_R} {\bf 1}_{x\in J_k\mbox{ or }\omega\notin A_{n_0}(x)}|L^{(R)}_\omega h(x)|^d
\right)^{\frac 1d} \\
\nonumber
&\leq& 
CR^2 \left(
\left[\frac{1}{|\disball_R|}
\sum_{x\in {\disball_R}} {\bf 1}_{x\in J_k\mbox{ or }\omega\notin A_{n_0}(x)}
\right]^{\frac 1{d+1}}
\left[\frac{1}{|\disball_R|}
\sum_{x\in {\disball_R}} |L^{(R)}_\omega h(x)|^{d+1}
\right]^{\frac d {d+1}}
\right)^{\frac 1d}.
\end{eqnarray}

On the event $A^{(1)}_R$,
\begin{align*}
\sum_{x\in {\disball_R}} {\bf 1}_{x\in J_k\mbox{ or }\omega\notin A_{n_0}(x)}
\le 
|J_k|+\sum_{x\in \disball_R}{\bf 1}_{\omega\notin A_{n_0}(x)}
\le 
(C\frac{k}{R}+2\alpha)|\disball_R|.
\end{align*}
Recall that $k=\sqrt R$. Hence, for $R>C/\alpha^2$,
\begin{equation}\label{eq:inD}
\frac{1}{|\disball_R|}
\sum_{x\in {\disball_R}} {\bf 1}_{x\in J_k\mbox{ or }\omega\notin A_{n_0}(x)}
\leq 3\alpha.
\end{equation}
Applying \eqref{eq:omega_laplacian_bnd} and Lemma \ref{lem:cont_T_1} on the event $A^{(3)}_R$, by \eqref{eq:inD} and \eqref{eq:from_mp},
\begin{equation}\label{eq:notinD}
\max_{x\in {\disball_R}}h(x)-\max_{x\in \partial^{(k)}{\disball_R}}h(x)
\le 
CR^2(3\alpha)^{1/(d+1)d} (C_F R^{-2})
=C_F \alpha^{1/(d+1)d}.
%
%
%
%
\end{equation}
Moreover, 
\begin{align*}
\max_{\partial^{(k)}{\disball_R}}h
&\le 
\gamma(R+k)^2+\max_{\partial^{(k)}{\disball_R}}H'\\
&\le 
\gamma(R+k)^2+\max_{
 \partial\disdomain R}H'
\stackrel{\eqref{eq:bdry_H'}}{\le} C_F(\epsilon+\tfrac{k}{R}),
\end{align*}
where in the second inequality we used the fact that $H'$ is $\omega$-harmonic in $\disball_{R+k}\setminus\disball_R$ and that an $\omega$-harmonic function achieves its maximum in the boundary.
%

Therefore, from \eqref{eq:from_mp}, \eqref{eq:inD} and \eqref{eq:notinD}, 
we have
\begin{equation*}
\max_{\disdomain{R}}H
\le 
{
\max_{\disdomain{R}}H'+\tfrac{C_F}{\sqrt R}
\le 
\max_{\disball_R}h+\epsilon C_F}
\le
C_F(\epsilon+\tfrac{k}{R})+C_F\alpha^{1/qd}.
\end{equation*}
Replacing $F$ by $-F$, similar arguments also give the same upper bound for
$\max_{\disdomain{R}}(-H)$.
Theorem~\ref{thm:quan_hom} now follows by taking $\alpha=\epsilon^{(d+1)d}$ and $R\ge \epsilon^{-2}+n_0$. 
\end{proof}

We now prove Lemmas \ref{lem:events_likely} and \ref{lem:cont_T_1}.

\begin{proof}[Proof of Lemma \ref{lem:events_likely}]
We start with estimating the probability of $A_R^{(2)}$. By 
Lemma \ref{lem:well_defined},
for every $x \in \disdomain{R}$,
\[
\annealedP^x(T_1 > R^{1/2}) \leq e^{-cR^{1/6}},
\]
and thus by Markov's inequality
\[
P\left( \big\{ \omega : \quenchedP_\omega^x(T_1 > R^{1/2}) \geq e^{-(\log R)^3} \big\} \right) < e^{(\log R)^3 - cR^{1/6}},
\]
and a union bound over all possible values of $x$ yields
\begin{equation}\label{eq:probA2}
P(A_R^{(2)}) \geq 1 -  Ce^{-R^{1/7}}.
\end{equation}

Next we estimate the probability of the event $A_R^{(1)}$. Note that the event $A_{n_0}(x)$ is determined by $\omega |_{y : \|y - x\| \leq n_0}$, and therefore
$\big(A_{n_0}(x_i)\big)_{i \in I}$ are independent events whenever $\forall_{i, j \in I} \|x_i - x_j\| > 2n_0$. For every $z \in [-n_0, n_0]^d$ we write $I_z = \big( z + (2n_0 +1) \Z^d \big) \cap \disdomain{R}$.
Then for every $z$ the events $\big(A_{n_0}(x_i)\big)_{i \in I_z}$ are independent, and each happens with probability greater than or equal to $1 - \alpha$. Therefore by Chernoff's inequality, for every $z \in [-n_0, n_0]^d$
\[
P \left(
\frac{\sum_{x \in I_z} {\bf 1}_{\omega\in A_{n_0}(x)}}{|I_z|} > 1-2\alpha
\right)
> 1 - e^{-C|I_z|}. 
\]

remembering that we have only finitely many choices of $z$, and that $|I_z| = \Theta (R^d)$ for every $z$, a union bound gives us
\begin{equation}\label{eq:probA1}
P(A_R^{(1)}) > 1 - e^{-CR^{d}}.
\end{equation}

The lemma now follows from \eqref{eq:probA2} and \eqref{eq:probA1}.

\end{proof}

\begin{proof}[Proof of Lemma \ref{lem:cont_T_1}]
For $W = R^{1/2}$ we write $\quenchedE_\omega^x[T_1] = E_1(x) + E_2(x)$ with $E_1(x) = \quenchedE_\omega^x[T_1 \cdot {\bf 1}_{T_1 \leq W}]$ and
$E_2(x) = \quenchedE_\omega^x[T_1 \cdot {\bf 1}_{T_1 > W}]$. We separately control
\[
\frac{1}{|\disball_R|}\sum_{x\in\disball_
R} \big( E_1(x) \big)^p
\ \ \ \ \ \ \ \ 
\mbox{and}
\ \ \ \ \ \ \ \ 
\frac{1}{|\disball_R|}\sum_{x\in\disball_
R} \big( E_2(x) \big)^p.
\]

To control the empirical $L^p$ norm of $E_1(x)$ we note that $E_1(x) \leq \quenchedE_\omega^x[T_1]$, and that, as in the proof of Lemma \ref{lem:events_likely}, $E_1(x_1)$ and
$E_1(x_2)$ are independent whenever $\|x_1 - x_2\| > 2W$. Therefore, following the same decomposition as in the proof of Lemma \ref{lem:events_likely}, for 
\[
K_1 = 2E\left[ \big(\quenchedE_\omega^x(T_1)\big)^p \right]
\]
we get
\begin{eqnarray}\label{eq:bnde1}
P\left[
\frac{1}{|\disball_R|}\sum_{x\in\disball_
R} \big( E_1(x) \big)^p > K_1
\right] < e^{-CR^{\delta'}}
\end{eqnarray}
with $\delta' > 0$ determined by $p$, by the choice of $W$ and by the dimension.
We now turn to control $E_2(x)$. We note that by Lemma
\ref{lem:well_defined}
$ 
E\left[ \quenchedE_\omega^x (T_1^2)  \right] < \infty,
$ 
and that $E\left[ \quenchedE_\omega^x ({\bf 1}_{T_1 > W})  \right] < e^{-cW^{1/3}} = e^{-cR^{1/6}}$.
Thus by Cauchy-Schwarz, 
\begin{equation}\label{eq:conte2}
E\left[ E_2(x) \right] \leq Ce^{-cR^{1/6}}.
\end{equation}

From \eqref{eq:conte2} we learn that 
\begin{equation}\label{eq:conte2av}
P\left[ \exists_{x \in \disball_R} E_2(x) > 1 \right] \leq \left| \disball_R \right| \cdot e^{-cR^{1/6}} \leq e^{-R^{1/7}}.
\end{equation}

From \eqref{eq:conte2av} we get that 
\begin{equation}\label{eq:e2p}
P \left[
\frac{1}{|\disball_R|}\sum_{x\in\disball_
R} \big( E_2(x) \big)^p > 1
\right]  \leq e^{-R^{1/7}}.
\end{equation}

The lemma now follows from \eqref{eq:bnde1} and \eqref{eq:e2p} once we remember that $\quenchedE_\omega^x(T_1)^p \leq p \big( E_1(x)^p + E_2(x)^p \big)$ for all $x$.

\end{proof}

One corollary of Theorem~\ref{thm:quan_hom} is particularly useful for us. This is Corollary~\ref{cor:exit} below, which is a very slight generalization of Corollary \ref{cor:quan_hom}.

\ignore{
Two corollaries of Theorem~\ref{thm:quan_hom} are particularly useful for us. The first one, Lemma \ref{lem:ann} below, bounds from below the minimum of a $\omega$-harmonic 
function on a large ball, using its minimum on a smaller ball. This lemma is an adaptation of Lemma 4.1 of \cite{KuoTru} to our case, and the proof is based on the idea in \cite{KuoTru} where 
we use a harmonic function and Theorem~\ref{thm:quan_hom} instead of the cutoff function of \cite{KuoTru}. The second one is Corollary~\ref{cor:exit} below, which controls the exit distribution from a ball.

\begin{lemma}\label{lem:ann}
Let $\tau<\sigma<1$. There exists $\gamma=\gamma(\tau,\sigma)>0$ and $C=C(\gamma,\sigma,P)$ such that for every $R$, with probability at least $1-C\exp(-R^{\delta})$, for every $g:\desball{R}{0}\to\R$ which is $\omega$-harmonic (Need to explain where exactly)
and non-negative,
\[
\min_{x\in\desball{\sigma R}{0}}g(x) \geq \gamma\min_{x\in\desball{\tau R}{0}}g(x).
\]
\end{lemma}


\begin{proof}
Let $m = \min_{x\in\desball{\tau R}{0}}g(x)$.

Let $\domain=\ball{1}{0} \setminus \overline {\ball{\tau}{0}}$, and let $f:\partial \domain\to\R$ be defined as
\[
f(x)=\left\{
\begin{array}{ll}
0 & \mbox{if } \|x\| = 1 \\
1 & \mbox{if } \|x\| = \tau
\end{array}
\right..
\]
Fix $R$. Let $F$, $G$ and $H$ be defined the same way as before the statement of Lemma \ref{lem:events_likely}.
Let $w=\min\{F(x):\|x\|\leq \sigma R\}$. 
We then take $\gamma=w/2$. Note that $w$ and $\gamma$ do not depend on $R$.
Let $\epsilon$ and $\alpha$ be so that $\BndSecDer{2}{f} \big(3\epsilon + C\cdot(3\alpha)^{\frac{1}{qd}}\big)<w/3$.
By Theorem~\ref{thm:quan_hom}, with probability greater than or equal to $1-C\exp(-R^{\delta})$,
\[
\max_{x\in\disdomain{R}} |F(x)-G(X)|<w/3.
\]
In particular, with the same probability, 
\[
\min_{\tau R <\| x\| < \sigma R} G(x) > \gamma
\]
We now consider the function $h=mG-g$. $h$ is $\omega$-harmonic on $\disdomain{R}$, and is non-positive on its boundary. Therefore
$h$ is nonpositive everywhere in $\disdomain{R}$, and for every $x\in \desball{\sigma R}{0} \setminus \desball {\tau R}{0}$,
\[
g(x) \geq g(x) + h(x) = mG(x) \geq \gamma m,
\]
and the lemma follows.

\end{proof}

} 

For any $A\subset \partial \cball 1$ and $R\ge 1$, we define
\begin{equation}\label{def: bdry-R}
\tilde A_R:=\{
x\in\partial\disball_R: x/|x|_2\in A
\}.
\end{equation}
For $x\in\R^d$, we let $\pbm^x$ denote the law of the Brownian motion with limiting covariance matrix $\CovMat$ and starting point $x$. We may describe an event without mentioning the underling Brownian motion. E.g, it should be clear what $\pbm^x(\mbox{exits the ball $\cball 1$ through }A)$ means.

\begin{corollary}\label{cor:exit}
Let $A\subseteq\partial \cball 1$ be open in the relative topology of $\partial \cball 1$.
Assume also that the boundary of $A$ w.r.t.\ the topology of $\partial \cball 1$ has measure zero w.r.t. the ($d-1$ dimensional) Lebesgue measure on $\partial \cball 1$.
For $x\in\overline{\cball 1}$, let 
\[
\ExitProb{A}(x)=\pbm^x(\mbox{exits $\cball 1$ through }A)
\]
For $\epsilon,r\in(0,1)$ and $R>0$, let 
\[
G(A, R,r,\epsilon)=
\left\{
\omega: \max_{x\in \disball_{rR}}|\ExitProb{A}(\tfrac{x}{R})-P_\omega^x(X_\cdot \mbox{ exits $\disball_R$ through }\tilde A_R)|\le\epsilon
\right\}.
\]
Then, there are constants $c, C$ depending on $A,r$ and $\epsilon$, such that for every $R$, 
\[
P(G(A,R,r,\epsilon))\ge 1-C\exp(-cR^\delta).
\]
\end{corollary}

\begin{proof}
We fix $\epsilon>0$. Recall the constant $n_1$ in Theorem~\ref{thm:quan_hom}. If suffices to prove the lemma for all $R\ge n_1(\epsilon^4, P)$.

Our proof consists of several steps.
\begin{enumerate}[Step 1.]
\item 
First, we will define two functions $F^{(1)}, F^{(2)}$ with smooth boundary data such that $F^{(1)}\le \ExitProb A\le F^{(2)}$.
For $A\subset\partial\cball 1$ and $\epsilon\in(0,1)$, we define  subsets $A_\epsilon^-, A_\epsilon^+$ of $\partial\cball{1}$  as
\begin{align*}
&A_\epsilon^-=\left\{x\in A: \dist(x, \partial\cball 1\setminus A)\ge\epsilon\right\},\\
&A_\epsilon^+=\left\{x\in\partial\cball 1:\dist(x,A)\le\epsilon\right\}.
\end{align*}
Clearly, $A_{2\epsilon}^-\subset A_\epsilon^-\subset A\subset A_\epsilon^+\subset A_{2\epsilon}^+$.
We can construct two smooth functions $f^{(\ell)}:\partial\cball 1\to[0,1]$, $\ell=1,2$ such that all of their $i$-th order partial derivatives have absolute values less than $\epsilon^{-i}$ for  $i=1,2,3$, and
\[
\left\{
\begin{array}{lr}
&f^{(1)}|_{A_{2\epsilon}^-}=1\\
&f^{(1)}|_{\partial\cball 1\setminus A_\epsilon^-}=0
\end{array}
\right.\qquad\mbox{ and }\qquad
\left\{
\begin{array}{lr}
&f^{(2)}|_{A_{2\epsilon}^+}=1\\
&f^{(2)}|_{\partial\cball 1\setminus A_{2\epsilon}^+}=0
\end{array}
\right..
\]
Then $f^{(1)}$ is supported on $A_\epsilon^-$ and $f^{(2)}$ is supported on $A_{2\epsilon}^+$.
Now for $\ell=1,2$, let $F^{(\ell)}: \overline{\cball 1}\to[0,1]$ be the solution of the Dirichlet problem
\[
\left\{
\begin{array}{lr}
\sum_{i,j=1}^d\CovMat_{ij}\partial_{ij} F^{(\ell)}=0 &\mbox{ in }\cball 1\\
F^{(\ell)}=f^{(\ell)} &\mbox{ on }\partial\cball 1.
\end{array}
\right.
\]
Note that $f^{(1)}\le 1_A\le f^{(2)}$ on $\partial\cball 1$ and so
\[
F^{(1)}\le \ExitProb A\le F^{(2)} \quad \mbox{ in } \overline{\cball 1}.
\]
Note also that by the definitions of $f^{(\ell)}$, we have for $i=1,2,3$ and $\ell=1,2$,
 \[
 \BndSecDer{i}{F^{(\ell)}}\le C/\epsilon^i,
 \]
 where $\BndSecDer{i}{F}$ denotes the supremum of the absolute values of all $i$-th order derivatives of $F$ over  $\cball 1$. 
Moreover, for $\ell=1,2$,
\begin{equation}\label{eq:F-perturb}
\sup_{\cball r}|\ExitProb A-F^{(\ell)}|\le \sup_{x\in\cball r}\pbm^x(\mbox{exits $\partial\cball 1$ from }A_{2\epsilon}^+\setminus A_{2\epsilon}^-) Ê\mathop{\longrightarrow}^{\epsilon \to 0} 0 
\end{equation}
\item 
Next, we will define two $\omega$-harmonic functions on $\disball_{R+1}$ whose boundary values  agree with that of $F^{(\ell)}$, $\ell=1,2$. 
Let $\disdomain{R+1}=\left\{x\in\disball_{R+1}: y\in\disball_{R+1}\mbox{ for all }y\sim x\right\}$. Note that $\overline{\disdomain{R+1}}=\disball_{R+1}$. For $\ell=1,2$, let $G_{R+1}^{(\ell)}:\overline{\disball_{R+1}}\to[0,1]$ be the solution of the Dirichlet problem
\[
\left\{
\begin{array}{lr}
L_\omega G_{R+1}^{(\ell)}=0 &\mbox{ in }\disdomain{R+1}\\
G_{R+1}^{(\ell)}=F_{R+1}^{(\ell)} &\mbox{ in }\partial\disdomain{R+1}.
\end{array}
\right.
\]
Recall that $F_{R+1}^{(\ell)}(x):=F^{(\ell)}(x/(R+1))$ for $x\in\overline{\cball {R+1}}$. Then, for any $R\ge n_1(\epsilon^4, P)$, by Theorem~\ref{thm:quan_hom}, with probability at least $1-C\exp(CR^\delta)$  we have
\begin{equation}\label{eq:Fell-Gell}
\max_{\disball_{R+1}}|F_{R+1}^{(\ell)}-G_{R+1}^{(\ell)}|\le 
\epsilon^4(\BndSecDer{2}{F^{(\ell)}}+\BndSecDer{3}{F^{(\ell)}})\le C\epsilon,
\quad \ell=1,2.
\end{equation}
\item 
We will show for all $R\ge n_1(\epsilon^4, P)\ge \epsilon^{-8}$, 
\begin{equation}
\label{eq:bdry-perturb}
F^{(1)}_{R+1}-c\epsilon
\le 
1_{\tilde A_R}
\le 
F^{(2)}_{R+1}+C\epsilon
\quad\mbox{ on }\partial\disball_R.
\end{equation}
First, for any $x\in\partial\disball_R\setminus\tilde A_R$, since $R\ge \epsilon^{-8}$, we have
$\dist(\tfrac x{R+1}, \partial\cball 1\setminus A^-_\epsilon)\le 1/R\le\epsilon^8$ and $\dist(\tfrac x{R+1},A_\epsilon^-)\ge C\epsilon$. Hence for  $x\in\partial\disball_R\setminus\tilde A_R$,
\[
1-F_{R+1}^{(1)}(x)\ge \pbm^{x/{R+1}}(\mbox{exits $\partial\cball 1$ at $\partial\cball 1\setminus A_\epsilon^-$})\ge 1-c\epsilon^7,
\]
which implies that $F_{R+1}^{(1)}-c\epsilon^7\le 1_{\tilde A_R}$ on $\partial\disball_R$. The first inequality of \eqref{eq:bdry-perturb} is proved.

Similarly, to obtain the second inequality of \eqref{eq:bdry-perturb}, notice that for any $x\in\tilde A_R$, $\dist(\tfrac{x}{R+1}, A_\epsilon^+)\le 1/R\le \epsilon^8$ and $\dist(\tfrac{x}{R+1},\partial\cball 1\setminus A_\epsilon^+)\ge C\epsilon$. Hence for $x\in\tilde A_R$, 
\[
F_{R+1}^{(2)}(x)\ge \pbm^{x/{R+1}}(\mbox{exits $\partial \cball 1$ at $A_\epsilon^+$})\ge 1-C\epsilon^7,
\]
which implies that $F_{R+1}^{(2)}+C\epsilon^7\ge 1_{\tilde A_R}$ on $\partial\disball_R$. Our proof of \eqref{eq:bdry-perturb} is now complete.
\item 
With \eqref{eq:bdry-perturb}, assuming \eqref{eq:Fell-Gell} we have
\[
G_{R+1}^{(1)}-c\epsilon
\le 1_{\tilde A_R}
\le G_{R+1}^{(2)}+C\epsilon\quad
\mbox{ on }\partial\disball_R
\] 
and so on $\overline{\disball_R}$,
\[
G_{R+1}^{(1)}-c\epsilon
\le 
P_\omega^x(X_\cdot \mbox{ exits $\disball_R$ through }\tilde A_R)
\le 
G_{R+1}^{(2)}+C\epsilon.
\]
This inequality, together with \eqref{eq:F-perturb} and \eqref{eq:Fell-Gell}, yields
\[
\max_{x\in \disball_{rR}}|\ExitProb{A}(\tfrac x R)-P_\omega^x(X_\cdot \mbox{ exits $\disball_R$ through }\tilde A_R)|\le C\epsilon.
\]
\end{enumerate}
Recalling that \eqref{eq:Fell-Gell} occurs with probability at least $1-C\exp(CR^\delta)$, the corollary is proved.
\end{proof}

\section{Percolation estimates}\label{sec:perc}

In this section we study connectivity properties of the balanced directed percolation at $\omega$.
%

The main results of this section are the following two propositions.
\begin{proposition}\label{prop:connect_outside_sink}
There exists $\ConnectExp>0$ depending only on the dimension and a constant $C<\infty$ depending on $P$, s.t.
\begin{equation}\label{eq:connect_outside_sink1}
P
\left[
\mbox{there exists } x \mbox { s.t. } 0\connom x, \ x\notin\thesink
\mbox{ and }
\|x\| = k
\right] < Ck^de^{-k^\ConnectExp}.
\end{equation}
\end{proposition}

\begin{proposition}\label{prop:antal_pist_sink}
For $x,y\in\Z^d$ we define the distance $\DistSink{x}{y}{\omega}\leq\infty$ as the length of a shortest $\omega$-path from $x$ to $y$. Note that in general
$\DistSink{x}{y}{\omega}$ may be different from $\DistSink{y}{x}{\omega}$. Then for
the same $\ConnectExp>0$ as in Proposition \ref{prop:connect_outside_sink}, and some constant $C<\infty$ depending on $P$, we have that for every $x$ and $y$
\begin{equation}\label{eq:connect_outside_sink2}
P\left[
\DistSink{x}{y}{\omega} > C\|x-y\|
\ ; \
x,y\in\thesink
\right]
<C\|x-y\|^de^{-\|x-y\|^\ConnectExp}.
\end{equation}
\end{proposition}

Unfortunately we need to provide different proofs for Proposition \ref{prop:connect_outside_sink} in two dimensions and in larger dimensions. The reason is that our high-dimensional proof uses the fact that Bernoulli percolation in $d-1$ dimensions 
has a non-trivial critical value, whereas our 2-dimensional proof relies heavily on planarity.

\subsection{Proof of Propositions \ref{prop:connect_outside_sink} and \ref{prop:antal_pist_sink} in three or more dimensions}
\begin{claim}\label{claim:sink_mon_dec}
Let $\numsinks(n)$ be the number of sinks in $[-n,n]^d$. Then a.s. there exists some $n_0$ such that $\numsinks(n+1)\leq \numsinks(n)$ for all $n>n_0$. In particular, the limit $\numsinks=\lim_{n\to\infty}\numsinks(n)$ exists a.s.
\end{claim}

\begin{proof}
Let $n_0$ be so large that
\begin{enumerate}
\item For all $n>n_0$, every sink in $[-n,n]^d$ has density at least $\Phi$.
\item For all $n>n_0$,
\begin{equation}\label{eq:boundary_small}
\frac{\left| [-n-1,n+1]^d \setminus [-n,n]^d \right|}
{\left| [-n-1,n+1]^d  \right|} < \frac \Phi 2.
\end{equation}
\end{enumerate}

Lemma \ref{lem:posdens}
guarantees that almost surely $n_0<\infty$.

Fix $n>n_0$. Then every sink in $[-n-1,n+1]^d$ intersects $[-n,n]^d$, and thus contains (at least one) sink in $[-n,n]^d$. Therefore $\numsinks(n+1)\leq \numsinks(n)$.
\end{proof}

We can also identify the value of the limit $\numsinks=\lim_{n\to\infty}\numsinks(n)$.

\begin{claim}\label{claim:num_sink_1}
$\lim_{n\to\infty}\numsinks(n) = 1.$
\end{claim}

\begin{proof}
By Proposition \ref{prop:unique}, the infinite sink is unique.
Assume for contradiction that $\numsinks>1$. Then there exists some $N>n_0$ (where $n_0$ is from the proof of Claim \ref{claim:sink_mon_dec})
such that $\numsinks(n)=\numsinks>1$ for all $n>N$. By induction on the previous argument, for every $n\geq N$, every sink in $[-n,n]^d$ intersects $[-N,N]^d$. Furthermore, every sink in $[N,N]^d$ is contained in a sink in $[-n,n]^d$, because
$\numsinks(n)=\numsinks(N)$. Let $x_1$ and $x_2$ belong to two distinct sinks in $[-N,N]^d$. Then the set of points (in $\Z^d$) reachable from $x_1$ is disjoint of the set of points reachable from $x_2$.
\ignore{
because if $y$ is reachable from both, it is reachable from both in some $[-n,n]^d$, and thus belong to both sinks, i.e. $x_1$ and $x_2$ are in the same sink in $[-n,n]^d$, and then
$\numsinks(n)<\numsinks(N)$.
} 
This stands in contradiction to 
Lemma \ref{lem:BD1459}
which implies that every point in the (infinite) sink is reachable from all points in $\Z^d$, and is thus reachable from both $x_1$ and $x_2$.
\end{proof}

As an immediate consequence of Claim \ref{claim:num_sink_1} we get the following Grimmett-Marstrand type lemma.

\begin{lemma}\label{lem:GM}
\[
\lim_{n\to\infty}P(\numsinks(n)=1)=1.
\]
\end{lemma}

Typically, the sink is ubiquitous in the cube. We make precise a weak sense of this statement.
\begin{lemma}\label{lem:sink_ubiq}
Fix $k$ and $\epsilon$. For all $n$ large enough, with probability greater than $1-\epsilon$ the following happens:
\begin{enumerate}
\item\label{item:unsink} There is a unique sink in $[-n,n]^d$.
\item\label{item:evrywhr} The sink intersects every sub-cube of side length $n/k$.
\end{enumerate}
\end{lemma}

\begin{proof}
We already know that Item \eqref{item:unsink} holds w.h.p. for large enough $n$. To see that Item \eqref{item:evrywhr} holds too, we note that it is enough to intersect $(4k)^d$
cubes of side length $n/2k$. If $n$ is large enough then each of them has a unique sink, and by the same induction as in the proof of Claim \ref{claim:sink_mon_dec}, the sink in $[-n,n]^d$
intersects each of them (note that their number does not grow with $n$).
\end{proof}

Another important fact is the following.
\begin{claim}\label{claim:bnd_pnt_in_sink}
There exist $N_0>N$ and $\ProbBndinSink>0$ such that for all $n>N_0$, for every $x\in\partial[-n,n]^d$, the probability that $x$ is in the sink of $[-n,n]^d$ is greater than
or equal to $\ProbBndinSink$.
\end{claim}

To prove Claim \ref{claim:bnd_pnt_in_sink} we first need some definitions.

For $z\in\Z^d$ and $n>0$, we write $\cubenz{n}{z}$ for the cube $[-n,n]^{d}+(2n+1)z$.

Now let $N$ be large enough for Lemma \ref{lem:sink_ubiq} to hold with $k=10$ and let $\epsilon=\epsilon(d)$ be small enough for what follows, and let $\n=10N$. Then for every $z$, with probability greater than $1-3d\epsilon$, we have that $\cubenz{\n}{z}$ contains a unique sink, that so does every nearest neighbour of $z$, and that the sinks are connected to each other. We call a cube satisfying these conditions good. Note that the event of goodness is independent beyond distance 2. Therefore, from the Liggett-Schonmann-Stacey theorem \cite{lss}, we Corollary \ref{cor:dminusine} below.

We write $\C(\cubenz{\n}{z})$ for the sink in the cube $\cubenz{\n}{z}$.Ê If there is more than one, we use some arbitrary scheme to choose one (in fact, we will only use this definition for good cubes, for which there is anyway only one sink).

\begin{corollary}\label{cor:dminusine}
If $\n$ is large enough,
the good cubes dominate Bernoulli site percolation which is supercritical for dimension $d-1$ (remember that $d\geq 3$).
\end{corollary}

We write $\BernClust$ for the infinite cluster of the percolation of good cubes. We write 
\[
\BernSink = \bigcup_{z\in\BernClust}\C(\cubenz{\n}{z}),
\]
and note that $\BernSink\subseteq\thesink$.

\begin{proof}[Proof of Claim \ref{claim:bnd_pnt_in_sink}]
Let $\n$ be as in Corollary \ref{cor:dminusine}, and assume w.l.o.g. that $(2 \n +Ê1)\ |\ (2n + 1)$. Break $\cubenz{n}{0}$ into cubes $\left(\cubenz{\n}{z} \right)_{z \in [-(2n + 1) / (2\n + 1) , (2n + 1) / (2\n + 1)]^d}$, and let $z_0$ be such that $x \in \cubenz{\n}{z_0}$.
Then with probability bounded away from zero in $n$, the good cubes in $[-(2n + 1) / (2\n + 1) , (2n + 1) / (2\n + 1)]^d \ \{z_0\}$ have a giant component neighbouring $z_0$, and conditioned on this event, with probability greater than or equal to $\kappa^{(2\N + 1)^d}$ there is a path from this component to the point $x$. Under event, whose probability is bounded away from zero in $n$, the point $x$ is in the sink.
\end{proof}


We now advance towards proving Proposition \ref{prop:connect_outside_sink}. 
We start by proving the bound for connection in the other direction.

\begin{lemma}\label{lem:connect_to_zero}
For $z\in\Z^d$, write $\GoingTo{\omega}{z}:=\{x:x\connom z\}$. Then
there exists $\ConnectExp>0$ such that
\begin{equation}\label{eq:connect_to_zero2}
P
\left[
|\GoingTo{\omega}{0}|>k
\mbox{ and }
\GoingTo{\omega}{z} \cap \BernSink =\emptyset
\right] < Ce^{-k^\ConnectExp}.
\end{equation}
\end{lemma}
\begin{proof}

Let
$\GoRen{\omega}{\n}:=\{z:\GoingTo{\omega}{0}\cap\cubenz{\n}{z}\neq\emptyset\}$.

We enumerate the set $\GoRen{\omega}{\n}$ using a breadth-first-search algorithm, as follows. First we arbitrarily enumerate $Z^d$, i.e. take a bijection $\pi:\N\to\Z^d$ with $\pi(0)= 0$. Then we define a sequence $(a_n)$ in $\Z^d\cup\{\infty\}$ inductively as follows.

We begin by writing $a_0=0$, and $\Cont_0:=\{x\connom 0 \mbox{ in }\cubenz{\n}{0} \}$, i.e. the set of points $x$ s.t. there is a directed path from $x$ to $0$ which is contained in $\cubenz{\n}{0}$.

Given $a_0,\ldots,a_n$ and $\Cont_n$, we define $a_{n+1}$ and $\Cont_{n+1}$ as follows. If $a_n=\infty$ then $a_{n+1}=\infty$ and $\Cont_{n+1}=\Cont_n$. Otherwise, take
\[
k_{n+1}=\inf
\left\{
k\ :\ \pi(k)\notin\{a_0,\ldots,a_n\} \mbox{ s.t. } \exists_{x\in\cubenz{\n}{\pi(k)}} \ x\connom 0 \mbox{ in } \cubenz{\n}{\pi(k)}\cup \left(\mathop\cup_{j=0}^n\cubenz{\n}{a_j}\right)
\right\}
\leq\infty,
\]
and 
$a_{n+1}=\begin{cases}
\pi(k_{n+1}) & k_{n+1}<\infty\\
\infty & k_{n+1}=\infty
\end{cases}$. 

\smallskip

We then take
\[
\Cont_{n+1} = \left\{
x\ :\ x\connom 0 \mbox{ in } \mathop\cup_{j=0}^{n+1}\cubenz{\n}{a_j}
\right\}
\]
if $a_{n+1} \neq \infty$ and $\Cont_{n+1} = \Cont_{n} $ if $a_{n+1} = \infty$.

We define a filtration $(\FF_n)_{n\geq 0}$ by 
\[
\FF_n=\sigma\big(a_0,\ldots,a_n,\omega|_{\mathop\cup_{j=0}^n\cubenz{\n}{a_j}}\big).
\]

Fix a coordinate $i=1,\ldots,d$, and define the stopping times 
\[
T_0=0\ ;\ T_{n+1}=\inf\left\{ j > T_n \ :\ 
\langle a_j,i\rangle \notin \{\langle a_k,i\rangle\ :\ k=0,\ldots,T_n\}
\right\}.
\]

We write $\Width{\omega}{i}=\sup\{n:T_n<\infty\}\leq\infty$.

Let $F_n$ be the event that  $\cubenz{\n}{a_n}$ is a good cube, $\Cont_n\cap \C(\cubenz{\n}{a_n})\neq\emptyset$ and that $a_n$ belongs to the infinite component in the hyperplane $\{z:\langle z,i\rangle = \langle a_n,i\rangle\}$. Note that if there exists $n$ such that $F_n$ occurs then 
$\GoingTo{\omega}{z} \cap \BernSink \neq \emptyset$.

Next we note that by Claim \ref{claim:bnd_pnt_in_sink} and Corollary \ref{cor:dminusine},
$
P[F_{T_n} | \FF_n\ ;\ \Width{\omega}{i}\geq n] \geq \rho
$
for some $\rho>0$.

Therefore for all $i$ and $n$,
\[
P\left(\Width{\omega}{i}\geq n \mbox{ and } 0\notin \thesink \right) < \rho^n.
\]

Note that if $\GoRen{\omega}{\n}>n$ then there exists $i$ such that $\Width{\omega}{i}>n^{1/d}$, which implies
the lemma with $\ConnectExp<1/d$.


\end{proof}

Using the fact that $\BernSink\subseteq\thesink$, and that if $x\in\thesink$ and $x\connom 0$ then $0\in\thesink$, we get the following corollary to Lemma \ref{lem:connect_to_zero}.

\begin{corollary}\label{cor:connect_to_zero}
there exists $\ConnectExp>0$ such that
\begin{equation}\label{eq:connect_to_zero1}
P
\left[
\left|x:x\connom 0\right| > k
\mbox{ and }
0\notin\thesink
\right] < Ce^{-k^\ConnectExp}.
\end{equation}
\end{corollary}

\begin{proof}[Proof of Proposition \ref{prop:connect_outside_sink}]
Using Corollary \ref{cor:connect_to_zero},
\begin{eqnarray*}
&&
P\left[
\mbox{there exists } x \mbox { s.t. } 0\connom x, \ x\notin\thesink
\mbox{ and }
\|x\| = k
\right] \\
&\leq&
\sum_{x:\|x\|=k}P\left[0\connom x, \ x\notin\thesink \right]
=\big|\partial[-k,k]^d\big| P\left[
\left|y:y\connom 0\right| > k
\mbox{ and }
0\notin\thesink
\right]\\
&\leq& Ck^de^{-k^\ConnectExp}.
\end{eqnarray*}

\end{proof}

\begin{proof}[Proof of Proposition \ref{prop:antal_pist_sink}]
First note that by The Antal-Pisztora theorem (Theorem 1.1 of \cite{antal_pisztora}) there exists $C$ and $\gamma$ such that
\begin{eqnarray}\label{eq:ap1}
P\left[
\DistSink{z}{w}{\omega} > C\|z-w\|
\ ; \
z,w\in\BernSink
\right]<C
e^{-\gamma\|z-w\|}.
\end{eqnarray}

Next we note that $|\GoingTo{\omega}{y}|=\infty$ for $y\in\thesink$, and thus by Lemma
\ref{lem:connect_to_zero} we get that 
\begin{eqnarray}\label{eq:ap2}
P\left[
\not\exists_{w\in\BernSink}
\DistSink{w}{y}{\omega}>k
\ ;\ y\in\thesink
\right]<Ck^de^{-k^{\ConnectExp}}.
\end{eqnarray}

We need to show that for all $x$,
\begin{eqnarray}\label{eq:ap3}
P\left[
\not\exists_{z\in\BernSink}
\DistSink{x}{z}{\omega}>k
\ ;\ x\in\thesink
\right]<
e^{-\gamma k}.
\end{eqnarray}

To this end we write $\ones$ for the vector
$(1,1,\ldots,1)\in\Z^d$
and then define a sequence $(x_n)_{n\geq 0}$ as follows. $x_0=x$. Given $x_n$ we choose $x_{n+1}$ as a nearest neighbor of $x$ satisfying
\begin{enumerate}
\item $\langle x_{n+1},\ones\rangle > \langle x_n,\ones\rangle$, and
\item $x_n\conneib x_{n+1}$.
\end{enumerate}
If there is more than one such neighbor, we apply some arbitrary scheme to choose one. The existence of $x_n$ is guaranteed by the balancedness of the environment.
Write $\D(j)$ for the union of the sinks in the boxes in the connected component of the Bernoulli percolation on 
$\{ z: \left| \langle z,\ones\rangle - j  \right| \leq 1 \}$.
Then, the events $\big(E_j=x_{3j\n}\in\D(j)\big)_{j\geq 1}$ are independent and of positive probability. \eqref{eq:ap3} follows.
\end{proof}

Proposition \ref{prop:antal_pist_sink} now follows from \eqref{eq:ap1}, \eqref{eq:ap2} and \eqref{eq:ap3}.

\subsection{Proof of Propositions \ref{prop:connect_outside_sink} and \ref{prop:antal_pist_sink} in two dimensions}
{
We write $x\sim y$ if $|x-y|_1=1$. A sequence $(x_i)_{i=0}^n$ is called a {\it path} if $x_0\sim x_1\ldots \sim x_n$. 
A subset $S\subset\Z^d$ is said to be {\it connected} if for any $x,y\in S$, there is a path $(x_i)_{i=0}^n\subset S$ with $x_0=x$ and $x_n=y$.}

\begin{proof}[Proof of Proposition~\ref{prop:connect_outside_sink}:]
{
First, we will show that the ``holes" outside of the sink are rectangles. 

Let $C$ be a connected (in the sense of $\sim$ ) component of $\Z^2\setminus\thesink$. For the sake of clarity, we color the unit square centered at $x$ (with sides parallel to the lattice) with white color if $x\notin\mathcal C$, and with blue color if $x\in\mathcal C$. Now consider the interface between the white and blue areas. The border of the blue area may consist of straight lines and angles with degrees $90^{\circ}$ or $270^{\circ}$. See Figure~\ref{fig:possibilities}.
However, case (C) in Figure~\ref{fig:possibilities} is impossible, since every point in the sink has at least two neighbors in opposite sides that are in $\mathcal C$.

Therefore, the border of $C$ consists of only straight lines and right angles. In other words, it is a rectangle.}

 The rectangle is finite w.p. 1, because the probability of an infinite line with no bond orthogonal emanating from
it is zero, and there are countably many such lines.

The probability for a given rectangle $L$ to be a connected component of $\Z^2\setminus\thesink$ is exponentially small in the length of the boundary of $L$. This proves Proposition \ref{prop:connect_outside_sink}
in two dimensions with $\ConnectExp=1/2$.
\end{proof}

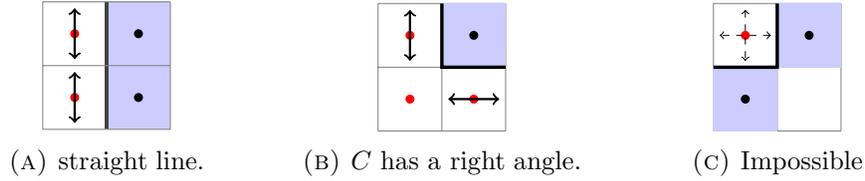
\begin{figure}
\begin{subfigure}[b]{0.30\textwidth}
\centering
\resizebox{\linewidth}{!}{
\begin{tikzpicture}[scale=1]
    \draw[fill=blue!20,draw=none] (0,0) coordinate (o) rectangle (1,-1); 
    \draw[fill=blue!20,draw=none] (o)  rectangle (1,1);
    \fill[red] (-.5, .5) coordinate (p1) circle (2pt);
    \fill[red] (-.5,-.5) coordinate (p2) circle (2pt);
        \fill (.5,-.5) coordinate (p3) circle (2pt);
    \fill (.5, .5) coordinate (p4) circle (2pt);
    \draw[line width=1,<->] (p1)+(0,.4)--++(0,-.4);
     \draw[line width=1,<->] (p2)+(0,.4)--++(0,-.4);
    \draw[line width =1.5] (0,-1)--(0,1);
       \draw[help lines] (-1,-1) grid (1,1);
    \node[minimum width=5cm] at (o){};
\end{tikzpicture}
}
\caption{straight line.} \label{fig:M1}
\end{subfigure}%
\begin{subfigure}[b]{0.30\textwidth}
\centering
\resizebox{\linewidth}{!}{
\begin{tikzpicture}[scale=1]
   \draw[help lines] (-1,-1) grid (1,1);
    \draw[fill=blue!20,draw=none]  (o)  rectangle (1,1);
    \fill[red] (p1) circle (2pt);
    \fill[red] (p2) circle (2pt);
    \fill (p4) circle (2pt);
    \fill[red] (p3) circle (2pt);
    \draw[line width =1.5] (0,1)--(o)--(1,0);
     \draw[line width=1,<->] (p1)+(0,.4)--++(0,-.4);
      \draw[line width=1,<->] (p3)+(.4,0)--++(-.4,0);
    \node[minimum width=5cm] at (o){};
\end{tikzpicture}
}
\caption{$C$ has a right angle.} \label{fig:M2}
\end{subfigure}%
\begin{subfigure}[b]{0.30\textwidth}
	\centering
	\resizebox{\linewidth}{!}{
		\begin{tikzpicture}[scale=1]
		\draw[help lines] (-1,-1) grid (1,1);
    	\draw[fill=blue!20,draw=none]  (o) rectangle (-1,-1); 
   	\draw[fill=blue!20,draw=none] (o)  rectangle (1,1);
    	\fill[red] (p1) circle (2pt);
    	    	\fill  (p2) circle (2pt);
    	\fill  (p4) circle (2pt);
    	     \draw[dashed,<->] (p1)+(0,.4)--++(0,-.4);
      \draw[dashed,<->] (p1)+(.4,0)--++(-.4,0);
      \draw[line width =1.5] (-1,0)--(0,0)--(0,1);
      \node[minimum width=5cm] at (o){};
	\end{tikzpicture}
	}
\caption{Impossible} \label{fig:M3}
\end{subfigure}%

\caption{Interface between $C$ and the sink $\mathcal C$. Red dots are points in the sink.}\label{fig:possibilities}
\end{figure}

\begin{proof}[Proof of Proposition~\ref{prop:antal_pist_sink}:]
When the dimension $d=2$, our proof of Proposition~\ref{prop:antal_pist_sink} consists of several steps.
In Steps~1--3, we define and estimate several geometric quantities in $\omega$. In Step~4 we will prove Proposition~\ref{prop:antal_pist_sink} using these geometric estimates.
\begin{enumerate}[Step 1.]
\item
First, we define a few terms. For a fixed environment $\omega$, 
the ES-stair (ES stands for east-south) is defined to be the infinite path starting from $o$ which goes first vertically upwards until it has the possibility to move right, and then it takes every opportunity to move right and moves downwards if a step to the right is not possible. The part of the ES-stair above the horizontal line is called the {\it ES-path}. See Figure~\ref{fig:ES-stair}. We  estimate the length $L^{ES}(\omega)$ of the ES-path. To do this, we
set $V_0=V_0(\omega)=\inf\{n\ge 0: \omega(ne_2,e_1)>0\}$, $H_0:=1$, and define recursively for $j\ge 0$,
\[
	H_j:=\inf\left\{n>0: \omega\big((n+H_0+\ldots+H_{j-1})e_1+(V_0-j)e_2, e_2\big)>0\right\}.
\]
We have the length of the ES-path
\[
L^{ES}(\omega)=2V_0+\sum_{j=0}^{V_0-1}H_j.
\]
Observe that $V_0,H_1,H_2,\ldots$ are independent (under $P$) geometric random variables, and $(H_j)_{j=1}^\infty$ are identically distributed. Hence we conclude that $L^{ES}$ has exponential tail. That is, for $x\ge 0$,
\[
P(L^{ES}>x)\le Ce^{-cx}.
\]
Similarly, we can define the EN-stair, EN-path and $L^{EN}$. The shorter one among the EN- and ES- path (If $L^{ES}=L^{EN}$ then take the EN-path) is simply called the {\it E-path}. The E-path has length
\[
L^E=L^E(\omega):=L^{EN}\wedge L^{ES},
\]
which also has exponential tail
\begin{equation}\label{eq:tail_of_L^E}
P(L^E>x)\le Ce^{-cx}, \qquad \forall x\ge 0.
\end{equation}

\item
The set of vertices that lie below (or on) the ES-stair and above (or on) the EN-stair is called the {\it E-bubble}, denoted by $B^E(\omega)$. In other words, $B^E(\omega)$ is the area enclosed by the EN- and ES- stairs. Clearly, $\# B^E(\omega)\le (L^{ES}+L^{EN})^2$ and so it has stretched exponential tail. That is, for $x\ge 0$,
\begin{equation}\label{eq:tail_of_bubble}
	P(\#B^E>x)\le Ce^{-c\sqrt x}.
\end{equation}
Here $\#B^E$ denotes the cardinality of the set $B^E$.

\item\label{step:tadpole}
Now we will define the east-tadpole. We denote the E-path by $p^E(\omega)$ and its end-point by $R^E(\omega)\in \N e_1$.
Set $R^E_0=0$ and define for $j\ge 0$
\[
	R^E_{j+1}:=R^E_j+E(\theta_{R^E_j}\omega).
\]	
In other words, $R^E_j$ is end-point of the concatenation $p_j^E(\omega):=\bigcup_{i=0}^{j-1}p^E(\theta_{R_i^E}\omega)$ of $j$ consecutive E-paths. For any $n\ge 0$, let
\[
M(n):=\inf\{i\ge 1:  R^E_i\cdot e_1\ge n\}
\]
and define the {\it east-tadpole} $T^E(n)=T^E_\omega(n)$ to be the union of the first $M(n)-1$ E-paths and the $M(n)$-th E-bubble. Namely,
\[
	T^E(n):=p^E_{M(n)-1}(\omega)\bigcup B^E(\theta_{R^E_{M(n)-1}}\omega).
\]
Since $M(n)\le n$, we have
\[
	\# T^E(n)\le \sum_{i=0}^{n-1}L^E(\theta_{R_i^E}\omega)+\max_{j=0,\ldots,n-1}\# B^E(\theta_{R^E_j}\omega),
\]
where the right side is a sum of 1-dependent random variables. It then follows from \eqref{eq:tail_of_L^E} and \eqref{eq:tail_of_bubble} that for $n\in\N$,
\[
	P(\# T^E(n)>Cn)\le Ce^{-c\sqrt n}.
\]
As the most important property of the tadpole, notice that for any $x\in T^E_\omega(n)$, if $o\connom x$, then we can find a $\omega$-path in $T^E(n)$ from $o$ to $x$. 
\item
Finally, we are ready to prove the theorem. Without loss of generality, assume that $x=o$ and $y=(y_1,y_2)$ with $y_1, y_2\ge 0$. Let $z=(y_1,0)$.  We define the {\it north-tadpole} $T^N(n)$ similarly as in Step~\ref{step:tadpole}.
By the remark at the end of Step~\ref{step:tadpole}, if $o\connom y$, then we can find a $\omega$-path in $T^E_\omega(y_1)\cup T^N_{\theta_z\omega}(y_2)$ from $o$ to $y$. Therefore,
\begin{align*}
	& P\left[
\DistSink{o}{y}{\omega} > C\|y\|
\ ; \
o\connom y
\right]\\
&\le 
P[\#T^E(y_1)+\# T^N_{\theta_z\omega}(y_2)\ge C(|y_1|+|y_2|)]\\
&\le 
Ce^{-c\sqrt n}.
\end{align*}
\end{enumerate}

\begin{figure} 
\tikzstyle{arrow}=[draw, -latex] 

\centering
\pgfmathsetseed{33}
  \begin{tikzpicture}[scale=.7]
    \draw[gray] (-2,0)--(6,0);
    \node (a) at (0,0) {o}[below left];
       \draw[white,pattern=north west lines, pattern color=gray!20] (a)--++(0,2)--++(2,0)--++(0,-1)--++(1,0)--++(0,-2)--++(-1,0)--++(0,-1)--++(-1,0)--++(0,-1)--++(-1,0)--(a);
    \draw[arrow, red]  (a)--++(0,2)--++(2,0)--++(0,-1)--++(1,0)--++(0,-1) coordinate (b);
    \draw[arrow, red, dashed] (b)--++(0,-1);
    \draw[arrow, blue]  (a)--++(0,-3)--++(1,0)--++(0,1)--++(1,0)--++(0,1)--++(1,0) coordinate (c);
   \draw[arrow, blue] (c)--++(1,0)--++(0,1);
  \end{tikzpicture}
\caption{The ES-(EN-) path is marked with a solid red (blue) line. The shaded region is an E-bubble.}\label{fig:ES-stair}
\end{figure}

\ignore{
\begin{figure}\label{fig:ET}
\tikzstyle{arrow}=[draw, -latex] 
\tikzset{
  declare function={%
    sign(\x) = (and(\x<0, 1) * -1) +
               (and(\x>0, 1) * 1);}}
\centering
\pgfmathsetseed{17}
  \begin{tikzpicture}[scale=.5]
    \draw[gray!20] (-2,0)--(20,0);
   \node (o) at (0,0) {o}[below left];
\foreach \j in {1,...,8}
{%
\pgfmathrandominteger{\v}{-3}{3}
 \pgfmathsetmacro{\absv}{abs(\v)}
 \pgfmathsetmacro{\d}{-sign(\v)} 
      \ifthenelse{\v=0}{\draw (o)--++(1,0) coordinate (o);}{
   \stair{\v}{\absv}{\d}{black}{}}
 }
 \draw[arrow] (o)--++(1,0) coordinate (a);
 \draw[fill=blue!20, fill opacity=0.2] (a)--++(0,2)--++(2,0)--++(0,-1)--++(1,0)--++(0,-2)--++(-1,0)--++(0,-1)--++(-1,0)--++(0,-1)--++(-1,0)--(a);
 \draw[arrow](a)++(3,0)--++(0,-1);
 \draw[arrow](a)++(2,-1)--++(1,0);
 \fill (a)++(2,0) coordinate (b) circle(3pt);
  \node[below] at (b) {$ne_1$};
  \end{tikzpicture}
  
  \caption{An east-tadpole $T^E(n)$}
\end{figure}
}

\end{proof}

\subsection{Relation to Harmonic functions}
From Propositions  \ref{prop:connect_outside_sink} and \ref{prop:antal_pist_sink} we can now prove a useful corollary regarding harmonic functions. Note that in uniformly elliptic cases
the following statement is trivial. We note that in the non uniformly elliptic setting there are genuinely $d$-dimensional, finite-range dependent counter examples to the following corollary.

\begin{corollary}\label{cor:perc_bnd_harm_func}
For every $\epsilon>0$ there exists $\nn$ and a constant $\PercHarmConst$ such that with probability greater than or equal to $1-\epsilon$, for every non-negative $\omega$-harmonic function $h$ 
on $[-2\nn,2\nn]^d$, we have
\begin{equation}\label{eq:Harnack_huge_const}
\max\{h(x) : x \in [-\nn,\nn]^d\} \leq \PercHarmConst \min\{h(x) : x \in [-\nn,\nn]^d\}.
\end{equation}
In particular,  if we denote by $\hitmstarx{\omega}{x}$ the hitting probability of the quenched random walk in $\partial {[-2\nn, 2\nn]^d}$,
namely $\hitmstarx{\omega}{x}(z) := \quenchedP_\omega^x(X_{T_\partial {[-2\nn, 2\nn]^d}} = z)$ for all $z \in \partial {[-2\nn, 2\nn]^d}$, then with probability
greater than $1 - \epsilon$,
\begin{equation}\label{eq:coupling_huge_const}
\max\{\| \hitmstarx{\omega}{x} - \hitmstarx{\omega}{y} \|_{\mbox{\bf TV}} : x , y \in [-\nn,\nn]^d\} \leq 1 - 1/{\PercHarmConst}.
\end{equation}
\end{corollary}

\begin{proof}
Fix $m$. 
Let $\xi=\xi(m)$ be so small that 
$
P\big[\exists_{z\in[-m,m]^d, e} \ \omega(z,e)\in(0,\xi) \big] < m^{2d}e^{-m^\ConnectExp}.
$
By Proposition \ref{prop:antal_pist_sink}, with probability at least $1-m^{2d}e^{-m^\ConnectExp}$, for every $x$ and $y$ in $[-3m/2,3m/2]^d$ we have
$\DistSink x y \omega < Cm$, and therefore 
\[
\max\{h(x) : x \in [-3m/2,3m/2]^d\cap\thesink\} \leq \xi^{-Cm} \min\{h(x) : x \in [-3m/2,3m/2]^d\cap\thesink \}.
\]
For a random walk in the random environment $\omega$, let $T_1$ be the first time it is in $\thesink$, and $T_2$ the first time it is outside $[-3m/2,3m/2]^d$.
We note that by Lemma \ref{lem:connect_to_zero} with probability at least $1-m^{2d}e^{-m^\ConnectExp}$, for every $x\in[-m,m]^d$, we have
$P_\omega^x(T_1<T_2)=1$ and thus for all $x\in[-m,m]^d$,
\[
\min\{h(x) : x \in [-3m/2,3m/2]^d\cap\thesink \} \leq h(x) \leq \max\{h(x) : x \in [-3m/2,3m/2]^d\cap\thesink\}.
\]

\eqref{eq:Harnack_huge_const} now follows if we take $\nn$ such that $3\nn^{2d}e^{-\nn^\ConnectExp}<\epsilon$ and $\PercHarmConst=\xi^{-C\nn}$. 
\eqref{eq:coupling_huge_const} follows from \eqref{eq:Harnack_huge_const} and the fact that $\hitmstarx{\omega}{x}$ is a harmonic function of $x$.

\end{proof}

\section{An oscillation inequality}\label{sec:osc}

The goal of this section is to obtain an oscillation estimate (Theorem~\ref{thm:osc}) for $\omega$-harmonic functions. 

\ignore{
The following is an immediate corollary of Lemma \ref{lem:ann}.
\begin{corollary}\label{cor:ann}
Let  $\Cube{R}=[-R,R]^d$. There exists $\gamma>1$ such that for every $R$, with probability $1-\exp(-CR^\delta)$, for every non-negative $\omega$-harmonic function $h$ on $\Cube{R}$, 
\[
\min \{h(x):x\in\Cube{R/2}\} \geq \gamma \min \{h(x):x\in\Cube{R/4}\}.
\]
\end{corollary}

Combining Corollary \ref{cor:ann} with Corollary \ref{cor:perc_bnd_harm_func} we get an a priori bound for the ratio between the minimum and the maximum in a large ball.

\begin{proposition}\label{prop:apribnd}
There exists $C$ and $\Apbnd$ such that for every $R$, with probability at least $1-\exp(-CR^{\delta})$,
every  non-negative $\omega$-harmonic function $u$ on $\desball{2R}{0}$ satisfies
\begin{equation}\label{eq:apribnd}
\max\{u(x):x\in\desball{R}{0}\} \leq R^\Apbnd \min\{u(x):x\in\desball{R}{0}\}
\end{equation}
\end{proposition}

To prove Proposition \ref{prop:apribnd}, we first need a percolation lemma.

\begin{lemma}\label{lem:percforapbnd}
Let $(p_n)$ be a sequence satisfying 
\begin{enumerate}
\item 
\[
\sum_{n=1}^\infty 3^{2dn}p_n < \frac 1{2d}, \mbox{ and}
\]
\item $p_n$ decays stretched exponentially with $3^n$.
\end{enumerate}
For every $n$, let $\CubeColl{n}$ be the collection of all cubes of side length $3^n$ in $\Z^d$ whose center is in
$3^{n-1}\Z^d$.
Let $\{X_A:A\in\cup_{n}\CubeColl{n}\}$ be an ensemble of Bernoulli variables satisfying:
\begin{enumerate}
\item $P(X_A=1) < p_n$ for $A\in\CubeColl{n}$,
\item If $A_1,\ldots,A_n$ are pairwise disjoint, then $X_{A_1},\ldots,X_{A_N}$ are independent.
\end{enumerate}
We say that $z\in\Z^d$ is open if there exists $A\in\cup_{n}\CubeColl{n}$ such that $z\in A$ and $X_A=1$.
Let $\mathcal C_0$ be the open cluster containing the origin (empty if the origin is closed). Then $P(|\mathcal C_0|>k)$ goes to zero
stretched exponentially with $k$.
\end{lemma}


\begin{proof}


{\red Structure of proof: Find largest cube in cluster. Then largest disjoint, and so on until there is nothing more. Then every cube has a neighbor, i.e. $A$ and $B$ neighbors if their distance smaller
than side length of maximal. Call this the skeleton of the cluster. developing the skeleton from zero is dominated by a sub-critical Galton-Watson tree.}

There exists $\alpha>0$ such that for every $k$, with probability greater than $\exp\big(-k^\alpha\big)$, there is no open box of size greater $k^{1/2d}$ with center at distance
up to $k^2$ from the origin. We call this event $A_k$, and write $B_K=\cap_{k\geq K}A_k$. Fix $K$, and write $C_K:=\{|\mathcal C_0|>K\}$.
Then,
$
P(C_K)\leq P(C_K|B_K)+P(B_K^c)
$
as $P(B_K^c)$ decays stretched-exponentially, it suffices to bound $P(C_K|B_K)$. Under the event $C_K$ there exists a point in $C_0$ at distance $k^{1/d}$ from the origin.
Let $\Path$ be an open path starting at zero and reaching distance $k^{1/d}$ from the origin, and let $(Q_j)_{j=0,\ldots,h}$ in $\cup_{n}\CubeColl{n}$ be the following sequence of cubes.
For every $z\in\Z^d$, let $Q(z)$ be the largest open cube containing $z$.
Then let $k_0=0$ and $Q_0:=Q(0)$, and  then 
\[
k_{j+1}=\inf\{k:Q(\Path_k)\cap\cup_{i=0,\ldots,j}Q_i=\emptyset\},
\]
and if 

\end{proof}

\begin{proof}[Proof of Proposition \ref{prop:apribnd}]
{\red Sketch, need to be improved}
Let $x_0\in\desball{R}{0}$ be where the maximum of $u$ in $\desball{R}{0}$ is attained. 
By Lemma \ref{lem:percforapbnd} and the maximum principle, with high probability, $x_0+\Cube{2\nn}$ satisfies the event in Corollary
\ref{cor:perc_bnd_harm_func}, and for every $1\geq k \geq \log_2 R$, the cube $x_0+\Cube{2\nn}$ satisfies the event in Corollary
\ref{cor:ann}. Therefore,
\[
\min\{u(x):x\in\desball{R}{0}\} \geq \PercHarmConst\gamma ^{\log_2 R} \max\{u(x):x\in\desball{R}{0}\}
\]
as desired.
\end{proof}

} 


For any finite subset $E\subset\Z^d$ and any function $u:E\to\R$, we define the {\it oscillation } of $u$ over the set $E$ by 
\[
\osc_E u:=\max_{x,y\in E}[u(x)-u(y)].
\]
For $z\in\Z^d$, $0<\alpha<1<\gamma<\infty$ and $R>0$,  we let $\OscEvent z R \alpha \gamma$ denote the event that every  $\omega$-harmonic function $f:\disball_{\gamma R}(z)\to\R$ satisfies
\begin{equation}\label{eq:osc}
\osc_{\disball_R(z)}f
\le 
\alpha
\osc_{\disball_{\gamma R}(z)} f. 
\end{equation}
\begin{theorem}\label{thm:osc}
There exist constants $\SizeOsc$ and $\OscConst<1$ such that $1-P(\OscEvent 0 R \OscConst \SizeOsc)$ decays stretched exponentially with $R$.
\end{theorem}

The proof uses techniques of probability coupling.
{
For any $R>1$, $z\in\Z^d$, define the hitting time of the {\it inner-boundary} of the ball $\disball_{R}(z)$ as
\[
\tau_{R,z}=\tau_{R,z}(X_\cdot)=\inf\{n\ge 0: X_n\in\partial(\Z^d\setminus\disball_R(z))\}.
\]
The underlying process of the stopping time $\tau_{R,z}$ should be understood {\it from the context}. For instance, the subscripts of $X_{\tau_{R,z}}$ and $Y_{\tau_{R,z}}$ represent two different stopping times $\tau_{R,z}(X_\cdot)$ and $\tau_{R,z}(Y_\cdot)$, respectively.

Our observation is that the oscillation estimate \eqref{eq:osc} will follow if for every $\omega\in\Omega$ and any $x,y\in \disball_R(z)$, there is a coupling of two paths $(X_n)$, $(Y_n)$ in $\disball_{\gamma R}(z)$ such that
\begin{enumerate}[ (a)]
\item The marginal distributions of $(X_n)$ and $(Y_n)$ are $P_\omega^x$ and $P_\omega^y$, respectively. With abuse of notation,  we use $P_\omega^{x,y}$ to denote the joint law of $(X_n, Y_n)_{n\ge 0}$.
\item $P_\omega^{x,y}(X_{\tau_{\gamma R,z}}=Y_{\tau_{\gamma R,z}})>1-\alpha$.
\end{enumerate}
Indeed, for any $\omega$-harmonic function $f$, $f(X_n)$ is a martingale under the quenched law $P_\omega^x$. 
Hence, by the optional stopping theorem, for any $x, y\in\disball_R(z)$,
\begin{align*}
f(x)-f(y)
&=E_\omega^{x,y}[f(X_{\tau_{\gamma R,z}})-f(Y_{\tau_{\gamma R, z}})]\\
&\le 
P_\omega^{x,y}(X_{\tau_{\gamma R,z}}\neq Y_{\tau_{\gamma R,z}})\osc_{\disball_{\gamma R}(z)}f\\
&\le \alpha\osc_{\disball_{\gamma R}(z)}f.
\end{align*}
}
We start by describing a multi scale structure.
\subsection{Multi scale structure}\label{sec:mltstr}

We fix three (large) parameters $R_0$, $\BallGrowth$ and $\BallScale$ and one (small) parameter $\errQb$ whose values will be determined later. We now say that $\BallScale$ needs to be an even number, and require that $\BallGrowth>100$. Further requirements will come later.

\label{page:partition}
Let $\Gamma=\{A_1,A_2,\ldots,A_k\}$ be a covering of $\partial\ball 1 0$ by closed sets intersecting only in their boundaries (in the $\partial\ball 1 0$ topology) and having relative boundaries with measure zero, such that the diameter of each of them
is smaller than $1/\BallGrowth^2$.

For a ball $B=\desball{R}{x}$ and a point $y\in\desball{R}{x}$, we denote by $\disribTile_\omega(B,y)$ the distribution on the set $\{1,\ldots,k\}$ with
$\disribTile_\omega(B,y)(j)=\quenchedP_\omega^y(T_{\partial\desball{R\BallGrowth}{x}}=T_{R\BallGrowth A_j})$. We write $\disribTile_\CovMat(B,y)$ for the distribution on the set $\{1,\ldots,k\}$ with
$\disribTile_\CovMat(B,y)(j)=\pbm^y(T_{\partial\desball{R\BallGrowth}{x}}=T_{R\BallGrowth A_j})$ where $\pbm$ is the distribution of Brownian Motion with covariance matrix $\CovMat$.

We now define the notion of goodness of a ball $\desball{R}{x}$.
\begin{definition}\label{def:good_ball}
\begin{enumerate}
\item If $R\leq R_0$, we say that the ball $B=\desball{R}{x}$ is good if it satisfies the event in 
\eqref{eq:coupling_huge_const}. 
\item If $R > R_0$, we say that the ball $B=\desball{R}{x}$ is good if for every $y\in\desball{R}{x}$, 
\[
\left\|
\disribTile_\omega(B,y)-\disribTile_\CovMat(B,y)
\right\|_{\mbox{TV}}<\errQb.
\]
\end{enumerate}
\end{definition}
 
 The following claim follows from 
 Corollaries \ref{cor:perc_bnd_harm_func} and \ref{cor:exit}.
 
 \begin{claim}\label{claim:ball_good_prob}
 There exists $\delta$ such that for every $x$ and $R$, the probability that the ball $\desball{R}{x}$ is good is at least $1-\exp(-R^\delta)$.
 \end{claim}
 
 We can now define our multi scale structure. We will recursively define the notion of an admissible ball. Claim \ref{claim:ball_good_prob} will then help
 us estimate the probability that a ball is admissible. We start with setting the scales. $R_0$ is given to us. We then define
 \[
 R_k:=R_{k-1}^\BallScale
 \]
 for $k\geq 1$.
 
 For now, we only define admissibility for balls of radius $R_k,k=0,1,2,\ldots$ To define admissibility, we first choose a parameter $\nu<\delta/\BallScale$,
 {
 where $\delta$ is as in Claim~\ref{claim:ball_good_prob}. }
 
 \begin{definition}\label{def:admis_ball}
 \begin{enumerate}
\item
A ball of radius $R_0$ is called \underline{admissible} if it is good.
\item
A ball of radius $R_k$, $k\geq 1$ is called \underline{admissible} if
\begin{enumerate}
\item Every sub ball of radius larger than $R_{k-1}$ is good, and
\item there are at most $R_k^{\nu}$ non-admissible sub balls of radius $R_{k-1}$.
\end{enumerate}
\end{enumerate}
\end{definition}
 
We now estimate the probability that a ball of radius $R_k$ is not admissible. We denote by $\admis(x,k)$ the event that $\desball{R_k}{x}$ is admissible.
\begin{lemma}\label{lem:prob_admis}
For every $x$ and $k$, the probability that $\desball{R_k}{x}$ is not admissible is bounded by $e^{-R_k^{\nu/2}}$.
\end{lemma}

\begin{proof}
For $k=0$ this follows from the fact that $\nu<\delta$. For $k\geq 1$ we prove the lemma by induction. Let $A$ be the event that there exists a sub ball of $\desball{R_k}{x}$ of radius larger
than $R_{k-1}$ which is not good, and let $B$ be the event that there are more than 
$R_k^{\nu}$
non-admissible sub balls of radius $R_{k-1}$. We estimate the probabilities of $A$
and $B$.

We start with estimating the probability of $A$. There are less than $(2R_k)^{d+1}$ sub balls of size greater than $R_{k-1}$ in $\desball{R_k}{x}$, and each of them is bad with probability less than
$e^{-R_{k-1}^\delta}<e^{-R_{k}^\nu}$. Thus
\[
P(A)\leq (2R_k)^{d+1} e^{-R_{k}^\nu}.
\]

We continue with estimating the probability of $B$. To this end we partition $\desball{R}{x}$ into $(\BallGrowth R_{k-1})^d$ subsets $L_1,L_2,\ldots,L_{(\BallGrowth R_{k-1})^d}$ such that for every $j$ and every $z,y\in L_j$, the events
$\admis(z,k-1)$ and $\admis(y,k-1)$ are independent.
For given $j$, we write
\[
U(j)=\sum_{z\in L_j}(1-{\bf 1}_{\admis(z,k-1)}).
\]
Then $U(j)$ is a binomial $(|L_j|,P(\admis(0,k-1)))$ random variable, and by the induction hypothesis is dominated by a binomial $((2R_k)^d,e^{-R_{k-1}^{\nu/2}})$ variable.
Let $\nu-\nu/2K<\nu'<\nu$ Then, For every $j$,
\begin{eqnarray*}
P\big(U(j)>R_k^{\nu'}\big) \leq (2R_k)^{dR_k^{\nu'}} e^{-R_k^{\nu'} R_{k-1}^{\nu/2}}\leq e^{-{ 0.9}R_k^\nu},
\end{eqnarray*}
and so
\[
P(B)\leq (\BallGrowth R_{k-1})^d e^{-{ 0.9}R_k^\nu}.
\]
Thus,
\[
P(\admis^c(x,k))\leq P(A)+P(B) \leq (2R_k)^{d+1} e^{-R_{k}^\nu} + (\BallGrowth R_{k-1})^d e^{-{ 0.9}R_k^\nu} \leq e^{-R_k^\nu/2}.
\]
\end{proof}

Before we continue with the proof of the oscillation inequality, we define admissibility also for balls whose radius is not exactly $R_k$ for some $k$.

\begin{definition}\label{def:admis_ball_gen_size}
Let $R>R_0$. Let $k$ be such that $R_k< R < R_{k+1}$. A ball of radius $R$ is called \underline{admissible} if
\begin{enumerate}
\item 
every sub ball of radius $R_k$ is admisible according to definition \ref{def:admis_ball}, and
\item 
every sub ball of radius greater than $R_k$ is good.
\end{enumerate}
\end{definition}

As a corollary of Lemma \ref{lem:prob_admis} we get the following corollary.

\begin{corollary}\label{cor:prob_admis}
For every $x$ and $R\geq R_0$, the probability that $\desball{R}{x}$ is not admissible is bounded by $e^{-R^{\nu/2\BallScale}}$.
\end{corollary}

\subsection{The coupling}
In this subsection we define a coupling that will be the main tool in proving the oscillation inequality. We start with a notion of a basic coupling, and will afterwards compose the coupling from many basic couplings.
\begin{definition}
Let $R\geq R_0$ and let $y$ and $z$ be points in $\desball{R}{x}$. The \underline{basic coupling} $\BasicCoupling{x,R;z,y}$ is a joint distribution of two times (i.e. natural numbers) $\tilde T_y$ and $\tilde T_z$ and two paths $(\tilde Y_1,\ldots,\tilde Y_{\tilde T_y})$ starting at $y$ and $(\tilde Z_1,\ldots,\tilde Z_{\tilde T_z})$ started at $z$ sampled as follows.
\begin{enumerate}
\item
If $R>R_0$, then $(\tilde Y_1,\ldots,\tilde Y_{\tilde T_y})$  and $(\tilde Z_1,\ldots,\tilde Z_{\tilde T_z})$ are sampled as random walks in the environment $\omega$, starting respectively at $y$ and $z$, with $\tilde T_y$ and $\tilde T_z$ being the respective stopping times of reaching $\partial\desball{R\BallGrowth }{x}$, where the two walks are coupled in a way that maximizes the probability that $\tilde Y_{\tilde T_y}$ and $\tilde Z_{\tilde T_z}$ are in the same element of $\{x+R\BallGrowth A_1,x+R\BallGrowth A_2\ldots,x+R\BallGrowth A_k\}$.
\item
if $R=R_0$ then $(\tilde Y_1,\ldots,\tilde Y_{\tilde T_y})$  and $(\tilde Z_1,\ldots,\tilde Z_{\tilde T_z})$ are sampled as random walks in the environment $\omega$, starting respectively at $y$ and $z$, with $\tilde T_y$ and $\tilde T_z$ being the stopping times of reaching $\partial\desball{R\BallGrowth }{x}$, where the two walks are coupled in a way that maximizes the probability that $\tilde Y_{\tilde T_y} = \tilde Z_{\tilde T_z}$.
\end{enumerate}
\end{definition}

On good balls, the basic coupling has a relatively good success probability, as evident by the following lemma, which follows immediately from the definition of good balls.
\begin{lemma}\label{lem:basic_coupling_success_probability}
Let $\desball{R}{x}$ be a good ball, and let $y,z\in\desball{R}{x}$.
\begin{enumerate}
\item If $R>R_0$ then
\[
\BasicCoupling{x,R;z,y}
\left(
\|\tilde Y_{\tilde T_y}-\tilde Z_{\tilde T_z}\|<R/\BallGrowth 
\right)>1 -\frac 1\BallGrowth  - 2\errQb.
\]
\item If $R=R_0$ then
\[
\BasicCoupling{x,R;z,y}
\left(
\tilde Y_{\tilde T_y} = \tilde Z_{\tilde T_z}
\right)>\PercHarmConst,
\]
where $\PercHarmConst$ is as in Corollary \ref{cor:perc_bnd_harm_func}.
\end{enumerate}
\end{lemma}

We now concatenate basic couplings, and get the following construction. Let $R$ be $R_0$ multiplied by a power of $\BallGrowth $, and let $y,z\in\desball{R}{x}$.
We will define $\CompCoupling{x,R;z,y}$ as a joint distribution of a random walk $(Y_n)$ starting at $y$, a random walk $(Z_n)$ starting at $z$, two sequences of stopping times $T_y^{(m)}$ and $T_y^{(m)}$, two sequences of points $(y_m)$ and $(z_m)$, and a sequence of balls $\big( \desball{R^{(m)}}{x_m} \big)$. 

To define $\CompCoupling{x,R;z,y}$, we first construct a coupling, and then take $\CompCoupling{x,R;z,y}$ to be its distribution.
We write $x_0=x$; $y_0=y$; $z_0=z$ and $R^{(0)}=R$, and also $T_y^{(0)}=T_z^{(0)}=0$. We also start the two random walks at the points $Y_0=y$ and $Z_0=z$.

Inductively, for $m=1,2,\ldots$, we now sample $y_m$, $z_m$, $x_m$, $R^{(m)}$, $T_y^{(m)}$, $T_z^{(m)}$ and $(Y_n)_{n=T_y^{(m-1)}}^{T_y^{(m)}}$ and $(Z_n)_{n=T_z^{(m-1)}}^{T_z^{(m)}}$, assuming that we already sampled $y_{m-1}$, $z_{m-1}$, $x_{m-1}$, $R^{(m-1)}$, $T_y^{(m-1)}$, $T_z^{(m-1)}$ and $(Y_n)_{n=0}^{T_y^{(m-1)}}$ and $(Z_n)_{n=0}^{T_z^{(m-1)}}$.

We sample $\tilde T_y$, $\tilde T_z$ and the two paths $(\tilde Y_1,\ldots,\tilde Y_{\tilde T_y})$ and $(\tilde Z_1,\ldots,\tilde Z_{\tilde T_z})$ according to $\BasicCoupling{{x_{m-1},R^{(m-1)};z_{m-1},y_{m-1}}}$ We then assign:
\begin{eqnarray*}
T_y^{(m)}:=T_y^{(m-1)}+\tilde T_y
\ ;\ 
T_z^{(m)}:=T_z^{(m-1)}+\tilde T_z
\ ; \ 
y_m=\tilde Y_{\tilde T_y}
\ ; \ 
z_m=\tilde Z_{\tilde T_z},
\end{eqnarray*}
as well as
\begin{eqnarray*}
Y_n=\tilde Y_{n - T_y^{(m-1)}} \mbox{ for } n=T_y^{(m-1)},\ldots,T_y^{(m)}  \mbox{   and   }
Z_n=\tilde Z_{n - T_z^{(m-1)}} \mbox{ for } n=T_z^{(m-1)},\ldots,T_z^{(m)}.
\end{eqnarray*}

to determine $x_m$ and $R^{(m)}$, we need to consider two different cases.
\begin{enumerate}
\item
If $R^{(m-1)}>R_0$, then if $\|y_m-z_m\|<R^{(m-1)}/\BallGrowth $ then we take $R^{(m)}=R^{(m-1)}/\BallGrowth $, and $x_m$ such that $y_m,z_m\in\desball{R^{(m)}}{x_m}$. Else we take $R^{(m)}=\BallGrowth R^{(m-1)}$ and $x_m=x_{m-1}$.
\item
if $R^{(m-1)}\leq R_0$, we stop the process (i.e. we stop the process one step after we reached a radius smaller than or equal $R_0$.). 
\end{enumerate}

Write $\FF_m$ for the $\sigma$-algebra generated by the environment $\omega$ and by $\big(y_k,z_k,x_k,R^{(k)},T_y^{(k)}, T_z^{(k)} \big)_{k\leq m}$ as well as
$(Y_n)_{n\leq T_y^{(m)}}$ and $(Z_n)_{n\leq T_z^{(m)}}$.

We define 
\[
L_m:=\frac{\log\big(R^{(m)}/R_0\big)}{\log \BallGrowth }.
\]
Note that $\big(L_m)_m$ is a random process whose step size is $1$. If $R^{(m)}>R_0$, and the ball $\desball{R^{(m)}}{x_m}$ is good, then by Lemma
\ref{lem:basic_coupling_success_probability},
\begin{equation}\label{eq:RWdomin}
\CompCoupling{x,R;z,y}\left(
L_{m+1}=L_{m}-1
\left| 
\FF_m \ ;\ \desball{R^{(m)}}{x_m}\mbox{ is good and }R^{(m)}>R_0
\right.
\right) > 1  -\frac 1\BallGrowth  - 2\errQb.
\end{equation}

At this point we choose $\BallGrowth $ and $\errQb$ so that $1  -\frac 1\BallGrowth  - 2\errQb>2/3$. For any $l>0$ we let $\StopLevel(l)$ be the stopping time
\[
\StopLevel(l):=\inf\{m:L_m\leq l\}.
\]
Then from \eqref{eq:RWdomin} we get domination by a biased one dimensional random walk which gives us the following estimate.
\begin{lemma}\label{lem:domin}
Let $z,y,x,R$ be so that $z,y\in\desball{R}{x}$, and let $l<\log R/\log \BallGrowth $. For every $j$,
\begin{equation}\label{eq:RWmax}
\CompCoupling{x,R;z,y}\left(
\sup\{L_m:m\leq \StopLevel(l)\} > L_0 + j \; ; \ \forall_{m\leq \StopLevel(l)}\desball{R^{(m)}}{x_m}\mbox{ is good }
\left| 
\omega
\right.
\right) \leq 2^{-j}
\end{equation}

\end{lemma}

Before we define the third and last coupling, we need an estimate regarding the hitting points of the two random walks in the the coupling $\CompCoupling{x,R;z,y}$.

\begin{lemma}\label{lem:hitpoint}
Fix $k$, and let $R=R^{(0)}=R_k^{\BallScale/2}$. Let $x,y,z$ be such that $y,z\in\desball{R}{x}$ and let $\omega$ be such that $\desball{R_{k+1}}{x}$ is admissible.  Let
\[
\tilde m := \inf\{m:R^{(m)}=R_k/\BallGrowth^k\}.
\]
Then for every $w\in \desball{R_{k+1}}{x}$,
\begin{equation}\label{eq:unfbnd}
\CompCoupling{x,R;z,y}
\left(
\|x_{\tilde m} - w\| < 10 R_k
\right) < R_k^{-\HitExp\BallScale}
\end{equation}
for some $\HitExp>0$ which is determined only by $P$ and $\BallGrowth$.
\end{lemma}

Note that using this lemma, once $P$ and $\BallGrowth$ are chosen, we can choose $\BallScale$ so that the exponent in \eqref{eq:unfbnd} is as small
as we like.

\begin{proof}
Let 
\[
\bar m=\inf
\left\{m:R^{(m)}\geq R_{k+1} \mbox{ or }  x_m\notin \desball{R_{k+1}/2}{x}  \right\}.
\]
First we show that 
\begin{equation}\label{eq:coupling_stays_in_ball}
\CompCoupling{x,R;z,y}(\bar m < \tilde m) < 2R_k^{-\rho_2\BallScale}
\end{equation}
for $\rho_2=\rho_2(P,\BallGrowth)>0$ which is specified below.

To see \eqref{eq:coupling_stays_in_ball}, we first note that by Lemma \ref{lem:domin}, for $\rho_1=\frac{\log 2}{2\log\BallGrowth}$,
\[
\CompCoupling{x,R;z,y}\left( \sup_{m<\tilde m} R^{(m)} \geq R_{k+1}  \right)  < 2^{-\frac{\BallScale}{2} \frac{\log R_k}{\log \BallGrowth}}=R_k^{-\rho_1\BallScale}.
\]
To estimate the probability that $x_m\notin \desball{R_{k+1}/2}{x}$, we estimate $E_{\CompCoupling{x,R;z,y}}\|x_{\tilde m} - x\|$ and use Markov's inequality. To estimate the expectation,
we note that $(L_m)$ is dominated by a random walk with a $(2/3,1/3)$ bias, and thus, for every $l$,

\begin{equation*}
E_{\CompCoupling{x,R;z,y}}\left(
\# m:L_m=l
\right) \leq
\begin{cases}
            2, & l\leq L_0\\
            2^{L_0-l}, & l\geq L_0
         \end{cases},
\end{equation*}
and 
\[
\|x_{m+1}-x_m\|\leq \BallGrowth R^{(m)}.
\]

Thus, noting that $\BallGrowth > 3$, we get

\begin{equation*}
E_{\CompCoupling{x,R;z,y}}\left(
\|x_{\tilde m} - x\| ; \max\{R^{(m)}:m\leq\tilde m\} \leq R_k^{\frac 34\BallScale}
\right)
\leq
R_k^{\frac 34\BallScale},
\end{equation*}
and we get 
\begin{eqnarray*}
\CompCoupling{x,R;z,y}\big(x_m\notin \desball{R_{k+1}/2}{x}\big) &=& \CompCoupling{x,R;z,y}\big(\|x_m - x\| >R_{k+1}/2\big)\\
&\leq& \CompCoupling{x,R;z,y}\left(\max\{R^{(m)}:m\leq\tilde m\} > R_k^{\frac 34\BallScale}\right) \\
&+& \frac
{E_{\CompCoupling{x,R;z,y}}\left(
\|x_{\tilde m} - x\| ; \max\{R^{(m)}:m\leq\tilde m\} \leq R_k^{\frac 34\BallScale}
\right)}
{R_{k+1}/2} \\
&\leq& R_k^{-\frac{\rho_1}{2}\BallScale} + R_k^{-\BallScale / 4} < R_k^{-\rho_2\BallScale}
\end{eqnarray*}
for appropriate $\rho_2>0$. Thus \eqref{eq:coupling_stays_in_ball} holds.

We now prove the statement of the lemma. By \eqref{eq:coupling_stays_in_ball}, with probability at least $1-2R_k^{-\rho_2\BallScale}$, until time $\tilde m$ the coupling only sees good balls. Therefore we can couple the random walk $(L_m)$ with a random walk $(\xi_m)$ such that
\begin{enumerate}
\item $(\xi_{m+1}-\xi_m)$ is an iid sequence satisfying $P(\xi_{m+1}-\xi_m=1)=1/3$ and $P(\xi_{m+1}-\xi_m=-1)=2/3$, and
\item For every $m$ we have $L_{m+1}-L_m \leq \xi_{m+1}-\xi_m$.
\end{enumerate}

We call $m$ a {\em regeneration} if 
\[
\sum_{n>m}\BallGrowth^{\xi_n-\xi_m} < 1.
\]
(Note that our definition of regenerations is quite different from the standard definitions in the literature). We use $\RegFlag (m)$ to denote the event that $m$ is a regeneration. We need to use the following fact, which says that there are plenty of regenerations. Note that Claim \ref{claim:many_reg}
below is a statement regarding biased simple random walks.
\begin{claim}\label{claim:many_reg}
There exist $\kappa>0$ and $\upsilon>0$ such that for every $N$,
\[
P\left(
\sum_{m=1}^N {\bf 1}_{\RegFlag (m)} <\kappa N
\right)
<e^{-\upsilon N}.
\]
\end{claim}

Using Claim \ref{claim:many_reg} we now finish the proof of  Lemma~\ref{lem:hitpoint}.
Note that if $m$ is a regeneration, then
\[
\|x_{\tilde m} - x_{m+1}\| <R^{(m)}/\BallGrowth.
\]

Therefore $\|w-x_{{\tilde m}}\|<R_k$ only if $\|w - x_{m+1}\| < 2 R^{(m)}/\BallGrowth$ for every regeneration $m$,
and this happens with probability bounded above by $C/\BallGrowth$ for every such $m$. The lemma follows.
\end{proof}

Using Lemma \ref{lem:hitpoint}, we estimate the probability that $\desball{R_k}{x_{\tilde m}}$ is not admissible, and see that if $\BallScale$ is large enough, then this probability is quite small.
Indeed, this probability is bounded by the number of non-admissible balls of radius $R^{(\tilde m)}$ with the uniform bound, obtained in Lemma \ref{lem:hitpoint}, on the hitting
probability of every ball. We get that
\begin{eqnarray}\label{eq:prob_hit_non_admis}
\CompCoupling{x,R;z,y}(\neg \admis(x_{\tilde m},k))
\leq R_k^{-\HitExp\BallScale} R_k^\delta.
\end{eqnarray}
For $\BallScale$ large enough, the power $\delta-\HitExp\BallScale$ is negative, which gives us a probability that is a negative power of $R_k$ to hit a non-admissible ball.

We can now proceed with the proof of Theorem~\ref{thm:osc}.

\begin{proof}[Proof of Theorem~\ref{thm:osc}]
In light of Lemma \ref{lem:hitpoint} and of \eqref{eq:prob_hit_non_admis}, we write $\NonAdmisExp:=\HitExp\BallScale-\delta$. Then, there exists a choice of our parameters such that 
$\NonAdmisExp>0$ and, \eqref{eq:prob_hit_non_admis} says that $\CompCoupling{x,R;z,y}(\neg \admis(x_{\tilde m},k))\leq R_k^{-\NonAdmisExp}$. We take $\SizeOsc = 2\BallGrowth$.

Let $R>0$ and assume that $R>R_0$ and that $\log(R/R_0)/\log \BallGrowth$ is an integer number. We later explain why these assumptions on $R$  do not limit the generality. Let $f:\desball{\SizeOsc R}{0}\to\R$ be $\omega$-harmonic function. Let $k_1$ be the largest number such that $R>R_{k_1}$. Let $y,z\in\desball{R}{0}$. 
Let  $(Y_n),(Z_n),(T_y^{(m)}),(T_z^{(m)}),(x_m),(R^{(m)})$ be sampled according to the coupling $\CompCoupling{x,R;z,y}$. We say that the coupling is successful if the following conditions are satisfied.
\begin{enumerate}
\item \label{item:cond_stay_in_big_ball} We require that 
\[
\big\{ Y_n:n>0 \big\}  \cup \big\{ Z_n:n>0 \big\}  \subseteq \desball{\SizeOsc R}{0}.
\]

\item \label{item:cond_admiss} for every $j=0,\ldots,k_1$, we write $\tilde m(j):=\inf\{m:R^{(m)}=R_j/\BallGrowth^j\}$. we require that for every $j$, the ball $\desball{R_j}{x_{\tilde m(j)}}$ is admissible.

\item \label{item:cond_stay_in_small_ball} We require that for every $j$ and every $m\geq \tilde m(j)$, we have $R^{(m)}<R_j$, and 
\[
\big\{ Y_n:n>T_y^{\tilde m(j)} \big\}  \cup \big\{ Z_n:n>T_z^{\tilde m(j)} \big\}  \subseteq \desball{R_j}{x_{\tilde m(j)}}.
\]

\item \label{item_cond_eventually_meet} We require that $Z_{T_z^{\tilde m(0)+1}} =Y_{T_y^{\tilde m(0)+1}}$.
\end{enumerate}
We call the event in Item \ref{item:cond_stay_in_big_ball} $A_1$, that in item \ref{item:cond_admiss} $A_2$, that in item \ref{item:cond_stay_in_small_ball} $A_3$ and that in item \ref{item_cond_eventually_meet} $A_4$. We write $A=A_1\cap A_2 \cap A_3 \cap A_4$. The main step in proving Theorem~\ref{thm:osc} is the following claim.
\begin{claim}\label{claim:walk_meet}
There exists $\MeetProb>0$ such that $\CompCoupling{x,R;z,y} \big(A|\admis(x,R)\big)>\MeetProb$ uniformly in $x,R;z,y$.
\end{claim}
We now see how Theorem~\ref{thm:osc} follows from Claim \ref{claim:walk_meet}, and then we prove the claim. For every point $p\in\partial\desball{\SizeOsc R}{0}$ let
$P_y(p)=\quenchedP_\omega^y\big(Y_{T_{\partial\desball{\SizeOsc R}{0}}}=p\big)$ and $P_z(p)=\quenchedP_\omega^z\big(Y_{T_{\partial\desball{\SizeOsc R}{0}}}=p\big)$.
By Claim \ref{claim:walk_meet}, if $\desball{R}{0}$ is admissible then
\[
\sum_{p\in\partial\desball{\SizeOsc R}{0}} \left| P_y(p)-P_z(p) \right| \leq 1-\MeetProb,
\]
or equivalently
\[
\sum_{p\in\partial\desball{\SizeOsc R}{0}}  \min\big(P_y(p),P_z(p)\big) \geq \MeetProb.
\]
We write $m_p=\min\big(P_y(p),P_z(p)\big)$, and then
\[
\sum_{p\in\partial\desball{\SizeOsc R}{0}}  \big(P_y(p)-m_p\big) = \sum_{p\in\partial\desball{\SizeOsc R}{0}}  \big(P_z(p)-m_p\big)  \leq 1-\MeetProb.
\]
Remembering that 
\[
f(y)=\sum_{p\in\partial\desball{\SizeOsc R}{0}} P_y(p)f(p)
\]
and equivalently for $z$, we get
\begin{eqnarray*}
f(z)-f(y) &=& \sum_{p\in\partial\desball{\SizeOsc R}{0}} f(p)P_z(p) - \sum_{p\in\partial\desball{\SizeOsc R}{0}} f(p) P_y(p) \\
&=&  \sum_{p\in\partial\desball{\SizeOsc R}{0}} f(p)\big(P_z(p)-m_p\big) - \sum_{p\in\partial\desball{\SizeOsc R}{0}} f(p) \big(P_y(p)-m_p\big) \\
&\leq& \max_{p\in\partial\desball{\SizeOsc R}{0}}f(p) \sum_{p\in\partial\desball{\SizeOsc R}{0}} \big(P_z(p)-m_p\big) \\
&-& \min_{p\in\partial\desball{\SizeOsc R}{0}}f(p) \sum_{p\in\partial\desball{\SizeOsc R}{0}} \big(P_y(p)-m_p\big) \\
&\leq& \big(1-\MeetProb\big) \left[\max_{p\in\partial\desball{\SizeOsc R}{0}}f(p) - \min_{p\in\partial\desball{\SizeOsc R}{0}}f(p)\right] 
\end{eqnarray*}
which, since $y$ and $z$ are arbitrary, proves the proposition with $\OscConst=1-\MeetProb$.
\end{proof}
We still need to prove Claim \ref{claim:walk_meet}
\begin{proof}[Proof of Claim \ref{claim:walk_meet}]
We bound the probability of $A_1^c$ exactly the same way \eqref{eq:coupling_stays_in_ball} is shown. By \eqref{eq:prob_hit_non_admis},
\[
\CompCoupling{x,R;z,y}
\left(
\admis \big(\desball{R_j}{x_{\tilde m(j)}}\big) \left| \ \bigcap_{h>j} \admis \big(\desball{R_h}{x_{\tilde m(h)}} 
\right.\right) > 1-R_j^{-\kappa},
\]
and as $R_k$ grows faster than exponentially in $k$, we get
\[
\CompCoupling{x,R;z,y}(A_2) = 
\CompCoupling{x,R;z,y}
\left(
\bigcap_j \admis \big(\desball{R_j}{x_{\tilde m(j)}}\big)
\right) \geq
1-\sum_{j=1}^\infty R_j^{-\kappa},
\]
which is, by the copice of $R_0$, as close as we want to $1$.

By \eqref{eq:coupling_stays_in_ball},
\[
\CompCoupling{x,R;z,y}(A_3) \geq
1-\sum_{j}R_j^{-\rho_2\BallScale}
\]
which, again, is close to $1$.

By Corollary \ref{cor:perc_bnd_harm_func},
\[
\CompCoupling{x,R;z,y}(A_4 | A_1\cap A_2\cap A_3)> C^\prime.
\]
Therefore $\CompCoupling{x,R;z,y}(A|\omega)$ is bounded away from zero in $R$ and in $\Omega$ satisfying $\admis(\desball{R}{x})$.
\end{proof}


Finally we provide a proof of Claim \ref{claim:many_reg}.

\begin{proof}[Proof of Claim \ref{claim:many_reg}]
We call $n$ a {\em renewal} if $\xi_m>\xi_n$ for all $m<n$ and $\xi_m<\xi_n$ for all $m>n$. Denote by $\tau_k$ the $k^{\mbox{\tiny th}}$ renewal. Then $(\tau_{k+1}-\tau_k)_{k\geq 1}$ is an i.i.d. sequence
and $\tau_1$, as well as $\tau_2-\tau_1$ have exponential tails. write $U_k=\xi_{\tau_k}$ and $V_k=\tau_{k-1}-\tau_k$. Then $(V_k)_{k\geq 1}$ is an i.i.d. sequence and there exists $\nu>0$ such that $P(V_1>l)<e^{-\nu l}$ for every $l$. In addition, $U_{k}-U_{k-1}\geq 1$ for every $k$. For $n=\tau_k$, we have
\begin{eqnarray*}
\sum_{n>m}\BallGrowth^{\xi_n-\xi_m} 
&=&
\sum_{j=k}^\infty M^{\xi_n}\left[\sum_{m=\tau_j+1}^{\tau_{j+1}}   M^{-\xi_m}    \right]\\
\leq
\sum_{j=k}^\infty M^{U_k}\left[V_j M^{U_j-1}    \right]
&\leq&
\sum_{j=k}^\infty V_j M^{k-j-1} 
\end{eqnarray*}
So, in particular, $\tau_k$ is a regeneration if $V_j M^{k-j-1}<2^{k-j-1}$ for every $j\geq k$. For $j\geq k$, we say that $j$ {\em influences} $k$, and denote it by $\infl{j}{k}$, if $V_j M^{k-j-1}\geq 2^{k-j-1}$. So $\tau_k$
is a regeneration if it is not influenced by any $j\geq k$. Let 
\[
\infed{j}=\sum_{k\leq j} {\bf 1}_{\infl{j}{k}}.
\]
Then $(\infed{j})_{j\geq 1}$ is an i.i.d. sequence with exponential tails, and for $M$ large enough we have $E(\infed{j})<1$. Thus by a large deviation estimate with probability exponentially close to $1$,
$\sum_{j=1}^J\infed{j})<CJ$ for some $C<1$, and under this event there are at least $(1-C)J$ regenerations.
\end{proof}

\section{Proof of the Harnack inequality}\label{sec:harnack}

In this section we prove Theorem \ref{thm:harnack}. We start with some preparation and notation.

\subsection{Preparation and notation}

Let $\Mesh>0$ and let $\Gamma=\{A_1,A_2,\ldots,A_k\}$ be a covering of $\partial\ball{1}{0}$ as in Subsection \ref{sec:mltstr} (Page \pageref{page:partition}), except that 
the diameter of the sets $A_1,A_2,\ldots,A_k$ is bounded by $\Mesh$. 
The exact value of $\Mesh$ will be specified later. Let $\errQb>0$, and for $j=1, \ldots, k$ let $\disribTile_\omega(B,y)(j)$ and 
$\disribTile_\CovMat(B,y)(j)$ be as in Subsection \ref{sec:mltstr}. For $0<\rho<1$ we write
\[
\DistribGood{z}{R}{\rho}{\Gamma}{\errQb}:=
\left\{
\forall_{y\in\desball{\rho R}{z}} \forall_{j=1,\ldots,k} \frac { \left| \disribTile_\omega(\desball{R}{z},y)(j) - \disribTile_\CovMat(\desball{R}{z},y)(j) \right| } { \disribTile_\CovMat(\desball{R}{z},y)(j) } < \errQb 
\right\}.
\]

Let $\OscConst$ and $\SizeOsc$ be such that by Theorem~\ref{thm:osc} the probability of ${\OscEvent{x}{r}{\OscConst}{\SizeOsc}}^c$ decays stretched exponentially.
Let $0<\PwrDwn<1/4$, and for $z\in\Z^d$ and $R>0$ let $\GoodEvents{z}{R}{\PwrDwn}$ be the following event:
\begin{equation}\label{eq:def_good_event}
\GoodEvents{z}{R}{\PwrDwn} :=
\bigcap_{x\in\desball{R}{z} ; R^{\PwrDwn}<r<R} \left (\OscEvent{x}{r}{\OscConst}{\SizeOsc} \cap \DistribGood{z}{R}{\rho}{\Gamma}{\errQb} \right)
\end{equation}

\begin{claim}\label{claim:probtov}
$1-P(\GoodEvents{z}{R}{\PwrDwn})$ decays stretched exponentially with $R$.
\end{claim}

\begin{proof}
This follows from Corollary~\ref{cor:exit}. and Theorem~\ref{thm:osc}.
\end{proof}

We will prove that a Harnack inequality for $\omega$-harmonic functions holds for every ball $\desball{R}{z}$ satisfying $\GoodEvents{z}{R}{\PwrDwn}$.

\subsection{Main lemma}
Let $\rdconst$ be the Harnack constant for harmonic functions in $\R^d$, 
and let $\omegaconst > \rdconst + 10\errQb$.

\begin{lemma}\label{lem:mainfharn}
Let $z$ and $R$ be such that the ball $\desball{2R}{z}$ satisfies $\GoodEvents{z}{2R}{\PwrDwn}$. Let $f:\desball{2R}{z} \to \R$ be non-negative and 
$\omega$-harmonic. Then
\[
\max_{x \in \desball{R}{z}} f(x) \leq \omegaconst \min_{x \in \desball{R}{z}} f(x).
\]
\end{lemma}

\begin{proof}
Assume for contradiction that there exist $x,y \in \desball{R}{z}$ such that $f(x)>\omegaconst f(y)$.
Let $\LessThanTwo<2$ be such that for every non-negative harmonic function $f$ on $\ball{\LessThanTwo}{0}$, we have
$\max_{\ball{1}{0}} f \leq (\rdconst + \errQb)\min_{\ball{1}{0}} f$.
We define a sequence of radii $r_0,r_1,r_2,\ldots,r_k$ as follows: $r_0=R$, $r_1=\errQb R/2$, and then $r_j=r_1/j^2$. We take $k$ to be the largest s.t.
$r_k>R^\PwrDwn$. Note that $k = k(R) > C R^{(1-\PwrDwn) / 2} > R^{1/3}$.

We now define a sequence of pairs of points $(x_0,y_0),(x_1,y_1),\ldots,(x_k,y_k)$. We will always have $\|x_j - y_j\|<2 r_j$ and that the distance of both $x_j$ and $y_j$ from $\partial \desball{2R}{0}$ is less than $4 r_j$.

We set a constant $\EpsComp=\omegaconst / (\rdconst - 2\errQb) - 1>0$.

At this point we can determine $\Mesh$, the mesh of the partition $\Gamma$. We take $\Mesh$ so that
\begin{equation}\label{eq:chooseMesh}
\Mesh \cdot \SizeOsc^{-\frac{\log 2 (\rdconst + 2\errQb)^2 \EpsComp^{-1}} {\log \OscConst}} < \min_{j = 1,\ldots} r_j/r_{j-1}
\end{equation}
where $\SizeOsc$ and $\OscConst$ are as in Theorem~\ref{thm:osc}.

We start by choosing $x_0=x$ and $y_0=y$. Then we need to explain how $x_{j+1}$ and $y_{j+1}$ are chosen, provided that we know $x_j$ and $y_j$. This explanation is postponed to after Claim \ref{claim:part_with_difference} below.

We can find a point $z_j$ (for $j=0$ we take $z_0=0$) such that $x_j$ and $y_j$ are both in $\desball{r_j}{z_j}$. Then for every set $A\in\Gamma$, we have
\begin{eqnarray}\label{eq:hitprobhrnk}
\nonumber
\big( \rdconst + \errQb \big)^{-1} &P_\omega^{x_j}& \left( T_{\partial \desball{\LessThanTwo r_j}{z_j}} = T_{z_j + \LessThanTwo r_j A} \right) \\
\nonumber
\leq
&P_\omega^{y_j}& \left( T_{\partial \desball{\LessThanTwo r_j}{z_j}} = T_{z_j + \LessThanTwo r_j A} \right) \\
\leq
\big( \rdconst + \errQb \big) &P_\omega^{x_j}& \left( T_{\partial \desball{\LessThanTwo r_j}{z_j}} = T_{z_j + \LessThanTwo r_j A} \right)
\end{eqnarray}

\begin{claim}\label{claim:part_with_difference}
If $f(x_j)/f(y_j) > \rdconst + 2\errQb$ then there exists $A\in\Gamma$ such that 
\begin{equation}\label{eq:part_with_difference}
\max_{z_j + \LessThanTwo r_j A} f
>
\frac{1}{\rdconst + 2\errQb}\frac{f(x_j)}{f(y_j)}  \min_{z_j + \LessThanTwo r_j A} f
\end{equation}
\end{claim}
We postpone the proof of Claim \ref{claim:part_with_difference}.

we now explain how $x_{j+1}$ and $y_{j+1}$ are chosen, provided that we know $x_j$ and $y_j$.

By Claim \ref{claim:part_with_difference} there is $A\in\Gamma$ such that \eqref{eq:part_with_difference} holds. Note that the diameter of $z_j + \LessThanTwo r_j A$
is bounded by $\Mesh r_j$. Thus we can find a point $z_{j+1}$ such that 
$z_j + \LessThanTwo r_j A \subseteq \desball{\Mesh r_j}{z_{j+1}}$. In particular, we get that 
\begin{equation*}
\max_{\desball{\Mesh r_j}{z_{j+1}}} f
>
\frac{1}{\rdconst + 2\errQb}\frac{f(x_j)}{f(y_j)}  \min_{\desball{\Mesh r_j}{z_{j+1}}} f.
\end{equation*}

Now note that by the choice of $\Mesh$ \eqref{eq:chooseMesh}, we get that
\[
r_{j+1}>\Mesh r_j \cdot \SizeOsc^{-\frac{\log 2 (\rdconst + 2\errQb)} {\log \OscConst}}.
\]
Thus, since the event $\GoodEvents{z}{R}{\PwrDwn}$ occurs, we get
\[
\osc_{\desball{r_{j+1}}{z_{j+1}}} f 
 \geq \OscConst ^ {\frac{\log 2 (\rdconst + 2\errQb)^2} {\log \OscConst}} \cdot  \osc_{\desball{\Mesh r_j}{z_{j+1}}} f 
=  2 (\rdconst + 2\errQb)^2 \cdot  \osc_{\desball{\Mesh r_j}{z_{j+1}}} f 
\]

We can then calculate
\begin{eqnarray}\label{eq:jddeuschel}
\nonumber
\frac{\max_{\desball{r_{j+1}}{z_{j+1}}} f } {\min_{\desball{r_{j+1}}{z_{j+1}}} f } &=& 1 + \frac{\osc_{\desball{r_{j+1}}{z_{j+1}}} f } {\min_{\desball{r_{j+1}}{z_{j+1}}} f } \\
\nonumber
\geq 1 + \frac{\osc_{\desball{r_{j+1}}{z_{j+1}}} f } {\min_{\desball{\Mesh r_j}{z_{j+1}}} f }  &\geq& 1 +  2 (\rdconst + 2\errQb)^2 \EpsComp^{-1} \frac{\osc_{\desball{\Mesh r_j}{z_{j+1}}} f } {\min_{\desball{\Mesh r_j}{z_{j+1}}} f } \\
\nonumber
= 1 + 2 (\rdconst + 2\errQb)^2 \EpsComp^{-1} \left[  \frac{\max_{\desball{\Mesh r_j}{z_{j+1}}} f } {\min_{\desball{\Mesh r_j}{z_{j+1}}} f } -1 \right] 
&\geq& 1 + 2(\rdconst + 2\errQb)\EpsComp^{-1} \left[  \frac{f(x_j) } {f(y_j)} - (\rdconst + 2\errQb) \right]\\
\end{eqnarray}

We now take $x_{j+1}$ and $y_{j+1}$ to be, respectively, points where the maximum and the minimum of $f$ in $\desball{r_{j+1}}{z_{j+1}}$ are obtained. It is easy to verify that the pairs $(x_j,y_j)$ satisfy the requirements above, namely that
$\|x_j - y_j\|<2 r_j$ and that the distance of both $x_j$ and $y_j$ from $\partial \desball{2R}{0}$ is less than $4 r_j$ for every $j$.

Write $\alpha_j = f(x_j)/f(y_j) - (\rdconst + 2\errQb)$. Then by our assumption, $\alpha_0 > \EpsComp$, and we can show that $\alpha_j > \EpsComp$ for every $j \geq 1$. Indeed, inductively using \eqref{eq:jddeuschel},
\begin{eqnarray*}
\alpha_{j+1} + (\rdconst + 2\errQb) \geq 1 + 2(\rdconst + 2\errQb)\EpsComp^{-1} \alpha_j \geq 1 + 2(\rdconst + 2\errQb) > (\rdconst + 2\errQb) + \EpsComp.
\end{eqnarray*}

Once we know that $\alpha_j > \EpsComp$, again using \eqref{eq:jddeuschel}, we get that
\begin{eqnarray*}
\alpha_{j+1} + (\rdconst + 2\errQb) \geq 1 + 2(\rdconst + 2\errQb)\EpsComp^{-1} \alpha_j 
\geq 
(\rdconst + 2\errQb) + 1 + (\rdconst + 2\errQb)\EpsComp^{-1} \alpha_j 
\end{eqnarray*}
which means that $\alpha_{j+1}>(\rdconst + 2\errQb)\EpsComp^{-1} \alpha_j>2\alpha_j$, and inductively, $\alpha_j > C2^j$. For $k$ as defined in the beginning of this proof, this means that
$f(x_k)/f(y_k)$ grows super-exponentialy with $\|y_k-x_k\|$, which contradicts Corollary \ref{cor:perc_bnd_harm_func}.

\end{proof}

\begin{proof}[Proof of Claim \ref{claim:part_with_difference}]
Write $\ConstClaimPart := \frac{1}{\rdconst + 2\errQb}\frac{f(x_j)}{f(y_j)} $.
Assume for contradiction that 
\[
\max_{z_j + \LessThanTwo r_j A} f
\leq
\ConstClaimPart  \min_{z_j + \LessThanTwo r_j A} f
\]
for every $A \in \Gamma$.

Then

\begin{eqnarray*}
f(x_j) &=& \sum_{z\in\partial \desball{\LessThanTwo r_j }{z_j}}f(z) P_\omega^{x_j}(X_{T \partial \desball{\LessThanTwo r_j }{z_j}} = z) \\
& \leq &
\sum_{A \in \Gamma} P_\omega^{x_j} \left( T_{\partial \desball{\LessThanTwo r_j}{z_j}} = T_{z_j + \LessThanTwo r_j A} \right)  \max_{z_j + \LessThanTwo r_j A} f\\ 
& \leq &
\ConstClaimPart \sum_{A \in \Gamma} P_\omega^{x_j} \left( T_{\partial \desball{\LessThanTwo r_j}{z_j}} = T_{z_j + \LessThanTwo r_j A} \right)  \min_{z_j + \LessThanTwo r_j A} f\\ 
&\leq&  (\rdconst +  \errQb) \ConstClaimPart \sum_{A \in \Gamma} P_\omega^{y_j} \left( T_{\partial \desball{\LessThanTwo r_j}{z_j}} = T_{z_j + \LessThanTwo r_j A} \right)  \min_{z_j + \LessThanTwo r_j A} f\\ 
& \leq &(\rdconst +  \errQb) \ConstClaimPart f(y_j) < f(x_j).
\end{eqnarray*}

and we reach a contradiction.

\end{proof}

\begin{proof}[Proof of Theorem \ref{thm:harnack}]
The theorem follows from Lemma \ref{lem:mainfharn} and Claim \ref{claim:probtov}.
\end{proof}

\ignore{
such that
\begin{enumerate}
\item $\diam(A_j)<\Mesh$ for $j=1,\ldots,n$,
\item The sets $A_j : j=1,\ldots,n$ are open in the relative topology, disjoint, and their boundaries have zero measure (w.r.t. Lebesgue measure on the $d-1$ dimenisonal sphere),
\item $\bigcup_{j=1}^n \overline{A_j}=\partial\ball{1}{0}$.
\end{enumerate}
}


\end{document}